\documentclass[11pt]{amsart}
\usepackage{graphpap}
\usepackage{amssymb}
\newtheorem{proposition}{Proposition}
\newtheorem{lemme}[proposition]{Lemma}
\newtheorem{theo}[proposition]{Theorem}

\newtheorem{remarque}[proposition]{Remark}

\numberwithin{equation}{section}
\numberwithin{proposition}{section}

\def\11{{\rm 1~\hspace{-1.4ex}l} }
\def\R{\mathbb R}

\def\Z{\mathbb Z}
\def\N{\mathbb N}

\begin{document}
\title[Quasi-invariant measures]{Quasi-invariant gaussian measures for one dimensional Hamiltonian PDE's}
\author{Nikolay Tzvetkov}
\address{
Universit\'e de Cergy-Pontoise,  Cergy-Pontoise, F-95000,UMR 8088 du CNRS}
\email{nikolay.tzvetkov@u-cergy.fr}
\begin{abstract} 
We prove the quasi-invariance of gaussian measures (supported by functions of increasing Sobolev regularity) under the flow of one dimensional Hamiltonian PDE's such as the regularized long wave (BBM) equation.
\end{abstract}

\maketitle
%
\section{Introduction}
\subsection{Motivation}
Our motivation for this work is twofold.
From one hand there is an extensive literature about the transport of Gaussian measures under nonlinear transformations (see e.g. \cite{Bog,Ra, ABC0, ABC, BW}).
These works treat either general nonlinear transformations close to the identity (see e.g. \cite{Ra}) or transformations generated by vector fields, under an exponential integrability assumption (see e.g. \cite{ABC}). It was however not clarified how much these results apply in the context of Hamiltonian PDE's. 

On the other hand, there is an extensive literature about invariant gaussian type measures (absolutely continuous with respect to gaussian measures) under the flows of Hamiltonian partial differential equations (see e.g.  \cite{LRS, B,B2, BB, BB2,BB3, BT1, BT2,BTT,Deng0,Deng, NORS,OH, OH2,Richards,S,Tz_fourier,TV,TVjmpa, QV, zh}). 
 In most of the cases the support of the measure consists of low regularity functions. Two exceptions are the KdV and the Benjamin-Ono equations where one can use the high order (i.e. controlling high order Sobolev norms) conservation laws in order to get invariant gaussian type measures supported by fairly smooth functions (see \cite{zh, TV, TVjmpa}). 
 
The existence of conservation laws controlling higher Sobolev norms is an exceptional event. 
Therefore, for Hamiltonian PDE's where conservation laws of high order are not available, it is not clear how the transport of gaussian measures, supported by functions of high Sobolev regularity, behaves under the corresponding Hamiltonian flow. Our goal here is to make a progress in this direction.
Namely, we will show that in the case of regularized long wave equations, gaussians measures supported by functions of arbitrary high Sobolev regularities are quasi-invariant by the flow of the corresponding equations. We recall that a measure $\mu$ on a space $X$ is called quasi-invariant under a transformation $\Phi:X\rightarrow X$ if its image under $\Phi$ is absolutely continuous with respect to $\mu$. 
\subsection{Derivation of a BBM type model}
If one considers shallow small amplitude long water waves, one obtains that the evolution of the water surfaces $u$ satisfies (formally) the equation
\begin{equation}\label{KdV}
\partial_t u+\partial_x u+\varepsilon_1 \partial_x^3u+\varepsilon_2 \partial_x(u^2)=O(\varepsilon_1^2+\varepsilon_2^2).
\end{equation}
In \eqref{KdV} $\varepsilon_1$ represents the square of the ratio between the depth of the fluid and the typical wave length while $\varepsilon_2$ represents the ratio between the wave amplitude and the depth (see e.g. \cite{Lannes}). 
Both $\varepsilon_1$ and $\varepsilon_2$ are small parameters. 
In the regime $\varepsilon_1\approx \varepsilon_2$ one takes into account both linear and nonlinear effects.
Therefore by neglecting the error in the right hand-side of \eqref{KdV} one ends up with the famous KdV equation.
The KdV equation has a highly oscillating linear part and a nonlinear part with a strong effect (derivative loss).
Since at first order $\partial_tu\approx -\partial_x u$, in \cite{BBM}  the authors introduced the model 
\begin{equation}\label{BBM}
\partial_t u+\partial_x u-\partial_t\partial_x^2u+\partial_x(u^2)=0
\end{equation}
as an alternative of the KdV model (in \eqref{BBM}, we dropped the $\varepsilon_{1,2}$ dependence). 
If we write \eqref{BBM} under the form
$$
\partial_tu+(1-\partial_x^2)^{-1}\partial_xu+(1-\partial_x^2)^{-1}\partial_x(u^2)=0,
$$
we observe that the model \eqref{BBM} has a slowly oscillating linear part coupled with a weak nonlinearity (with smoothing of degree one). 
Despite of this difference with respect to KdV the model \eqref{BBM} is supposed to describe a similar balance between linear and nonlinear effects.

One may naturally consider the following generalization (generalized dispersion) of the KdV equation 
\begin{equation}\label{gKdV}
\partial_t u+\partial_x u-|D_x|^\gamma \partial_xu+\partial_x(u^2)=0.
\end{equation}
For $\gamma=2$, we recover the KdV model but for  
$\gamma=1$ one gets the Benjamin-Ono equation which is a model that can be derived similarly to KdV but in the context of internal waves. 

Following the same argument as for deriving \eqref{BBM}, we end up with the following generalization of \eqref{BBM}
\begin{equation}\label{BBM-gamma}
\partial_t u+\partial_t|D_x|^\gamma u+\partial_x u+\partial_x(u^2)=0.
\end{equation}
For $\gamma=2$ we recover \eqref{BBM} while for $\gamma=1$ we deal with a Benjamin-Ono type model. 
The goal of this work is to study \eqref{BBM-gamma} with initial data distributed by gaussian measures in Sobolev spaces of an arbitrary regularity. 
Observe that, at least formally, if we multiply \eqref{BBM-gamma} by $u$ and integrate in $x$, we obtain that a Sobolev type norm of order $\gamma/2$ of $u$ is conserved by \eqref{BBM-gamma}.
This global information is the only useful a priori bound for \eqref{BBM-gamma} we are aware of. It will play an important role in the analysis below. 
\subsection{Statement of the results}
We consider \eqref{BBM-gamma}, posed on the one dimensional torus. 
Since the $x$ mean value is preserved by  \eqref{BBM-gamma}, we shall consider \eqref{BBM-gamma} as a dynamical system on the Sobolev spaces of zero $x$ mean value functions (equivalently functions having vanishing zero Fourier coefficient). We denote by $H^s$ the Sobolev space of zero mean functions (see the notation section below for a precise definition). 
The next statement shows that \eqref{BBM-gamma} defines a dynamical system on $H^s$, $s\geq \gamma/2$.
\begin{proposition}\label{l1}
Let $\gamma>1$ and $\sigma\geq \gamma/2$. Then for every $u_0\in H^\sigma$ there is a unique global solution of \eqref{BBM-gamma} in $C(\R;H^\sigma)$. 
Moreover, if we denote by  $\Phi(t)$ the flow of \eqref{BBM-gamma} then for every $t\in\R$, $\Phi(t)$ is a continuous bijection on $H^\sigma$. 
\end{proposition}
Once this result is established one may naturally ask qualitative questions of the global behavior of $(\Phi(t))_{t\in\R}$ as a dynamical system on $H^\sigma$.
As already mentioned, in this work we will study the transport of some gaussian measures by $\Phi(t)$ (for $\sigma$ not necessary small). 

We next introduce these measures.
Let $s\geq 1$ be an integer. Denote by $\mu_s$ the gaussian measure induced by the random Fourier series 
$$
\varphi_s(\omega,x)=\sum_{n\neq 0}\frac{g_n(\omega)}{|n|^{s+\gamma/2}}e^{inx},
$$
where $g_n=\overline{g_{-n}}$ and $(g_n)_{n>0}$ is a system of i.i.d standard complex gaussians, i.e. 
$$
g_n=\frac{1}{\sqrt{2}}(h_n+il_n),
$$
where $h_n,l_n\in {\mathcal N}(0,1)$ are independent.   
Strictly speaking the measure $\mu_s$ also depends on $\gamma$ but we do not explicit this dependence. 
For $\gamma>1$ the measure $\mu_s$ can be seen as a gaussian measure on $H^s$.
Therefore thanks to Proposition~\ref{l1}  for every $\gamma>1$ the flow $\Phi(t)$ is defined $\mu_s$ almost surely, provided $s\geq \gamma/2$. 
Our main goal is to prove the following statement.
\begin{theo}\label{main}
Let  $\gamma >4/3$.
Then for every integer $s\geq \gamma/2$ the measures $\mu_s$ is quasi-invariant by the flow $\Phi(t)$, for every $t\in\R$. 
\end{theo}
We also obtain some quantitative bounds on the densities of the transported measures (see e.g. Lemma~\ref{lyon} below).
However the present information these bounds give on the densities at time $t$ seems quite weak to be useful for giving new long time bounds on the solutions a.s with respect to $\mu_s$ (see e.g.  Remark~\ref{growth} below).
It would be interesting to improve on these bounds. For that reason, we decided to keep the quantitative part of our argument, hoping that it may be of some interest
in eventual further developments.  

We did not try to optimize the restriction $\gamma>4/3$, our goal being to achieve a value of $\gamma$ smaller than $3/2$ which allows to go beyond the Cameron-Martin
threshold (see also Proposition~\ref{CM} below). The assumption that $s$ is an integer is not essential and most probably can be removed. 

The measure $\mu_s$ can be seen as a normalized version of the formal object 
$$\exp(-\|u\|_{H^{s+\gamma/2}}^2)du.$$ 
For $\gamma>1$ the triple $({\rm Id}, H^{s+\gamma/2},H^s)$ forms a Wiener space and $\mu_s$ is the standard Gaussian measure on $H^s$ with variance parameter $1$. The space $H^{s+\gamma/2}$ is the canonical Hilbert space (the so called Cameron-Martin space) in this construction but the space $H^s$ may be replaced by any $H^\sigma$ with $\sigma<s+\frac{\gamma}{2}-\frac{1}{2}$.

In the case $s=0$ thanks to the conservation of 
the Sobolev norm of order $\gamma/2$ one can
get the invariance of the measure $\mu_0$, by employing the well-established methods of invariance of Gibbs measures, at least for $\gamma\geq 2$ (see \cite{ASdS}
for the case $\gamma=2$). The extension to some values of $\gamma<2$ would require some elaborations 
on the local in time analysis in the proof of  Proposition~\ref{l1} (to obtain the existence of the dynamics, locally in time, on the support of the measure $\mu_0$).
\subsection{Comparison with Cameron-Martin type of results} 
It is instructive to compare the result of Theorem~\ref{main} with the Cameron-Martin theorem \cite{CM} and a result by Ramer \cite{Ra}.
Denote by $S(t)$ the free evolution associated to \eqref{BBM-gamma}, i.e.
$$
S(t)=\exp(-t(1+|D_x|^\gamma)^{-1}\partial_x)\,.
$$
Then thanks to the Duhamel formula, the map $\Phi(t)$ can be written as
$$
\Phi(t)(u_0)=S(t)(u_0)+(\big(\gamma-1){\rm \,\,smoother\,\, part\,\, depending\,\,on\,\,}u_0\big).
$$
The free evolution $S(t)$ is preserving the Sobolev regularity and the measure $\mu_s$ (thanks to the invariance of the gaussians by rotations).
Therefore the Cameron-Martin theorem implies that $\mu_s$ is quasi-invariant by maps of the form 
$$
S(t)(u_0)+\big((\frac{1}{2}+\varepsilon){\rm \,\,smoother\,\, part\,\, {\bf independent}\,\, of\,\,}u_0\big),
$$
for some  $\varepsilon>0$. 
Therefore in the range $\gamma\in(4/3,3/2]$ the result of Theorem~\ref{main} goes beyond the naive intuition dictated by the Cameron-Martin theorem.
We find this phenomenon interesting (see also Proposition~\ref{CM} below).

Let us next compare the result of Theorem~\ref{main} with a result by Ramer.
Using once again the above mentioned properties of $S(t)$, one gets that the result of Ramer implies  
that  $\mu_s$ is quasi-invariant by maps of the form 
$$
S(t)(u_0)+\big((1+\varepsilon){\rm \,\,smoother\,\, part\,\, {\bf which\,\, may\,\, depend}\,\, on\,\,}u_0\big),
$$
for some $\varepsilon>0$.
Therefore for $\gamma>2$ the result of Theorem~\ref{main} follows from the work by Ramer \cite{Ra}, i.e. the BBM model \cite{BBM} 
is the border line (but not covered by \cite{Ra}). 
The assumption $\gamma>2$ is done in order to assure the Hilbert-Schmidt property of the perturbation imposed in \cite{Ra} (see Section~\ref{Ramm} below).
The work of Ramer deals with general maps and in this setting the $1+\varepsilon$ regularization condition (in dimension one) looks optimal.
The reason for which Theorem~\ref{main} goes beyond the result of \cite{Ra} is that we deal with very particular non-linear maps induced by Hamiltonian flows.
Here our work is close in spirit to the articles by Cruzeiro \cite{ABC0, ABC} which prove abstract results concerning the existence of quasi-invariant measures under the flows of  (not necessarily  smooth) vector fields. In the work of Cruzeiro the existence of the dynamics (the analogue of Proposition~\ref{l1}) is already a non-trivial issue (see also \cite{BW, AF} for more recent works). Concerning the quasi-invariance statement in \cite{ABC}, it is done under an exponential integrability assumption of the divergence of the corresponding vector field.  One may wish to see the result of Theorem~\ref{main} as an instance where such an integrability condition is checked 
"in practice".   

We end the discussion about the comparison between Theorem~\ref{main} and Cameron-Martin type of results by the following statement.
\begin{proposition}\label{CM}
Let $\gamma\in (4/3,3/2)$.
Consider the linear PDE
\begin{equation}\label{BBM-gamma-linear}
\partial_t u+\partial_t|D_x|^\gamma u+\partial_x u+\partial_x(h)=0,
\end{equation}
where $h\in H^\sigma$ for some $\sigma<s+\frac{\gamma}{2}-\frac{1}{2}$ is fixed.
Suppose that $h\notin H^{s+\frac{\gamma}{2}-\frac{1}{2}}$.
Denote by $\Sigma(t)$ the (well-defined) flow of \eqref{BBM-gamma-linear}. 
Then for $t\neq 0$ the transport of $\mu_s$ by $\Sigma(t)$ is a measure singular with respect to $\mu_s$.
\end{proposition}
In Proposition~\ref{CM}, the fixed function $h$ is supposed to have the typical regularity on the support of $\mu_s$. 
In other words, if $u$ is a solution of  \eqref{BBM-gamma} with data on the support of $\mu_s$ then we take $h$ with the regularity of $u^2$ obtained by
the deterministic estimates of Section~2 below (and not more).  
Therefore,  for $\gamma\in (4/3,3/2)$, 
the result of Theorem~\ref{main} seems to go beyond a Cameron-Martin type result and it relies on a  "regularization property" of the flow associated with 
\eqref{BBM-gamma}.
\subsection{Organization of the paper}
The remaining part of this manuscript is organized as follows.
We complete this introduction by introducing some notations. 
In Section~2 we prove the existence of the dynamics and some useful approximation properties.
In Section~3 we obtain the result of Theorem~\ref{main} for $\gamma>2$ as a consequence of \cite{Ra}.
Next, in Section~4 we establish a useful infinite dimensional change of variables formula. In Section~5 we 
get the suitable for our purposes (deterministic) energy estimate.
In Section~6 we establish the averaging with respect to $\mu_s$ properties, needed for our analysis. In Section~7 we establish the measure evolution property by an argument in the spirit of the proof of the global regularity for the $2d$ Euler equation. In Section~8 we complete the proof of the main result by some basic measure theory considerations. Finally, in Section~9 we prove Proposition~\ref{CM} as a simple consequence of the Cameron-Martin argument. 
\subsection{Notation}
If a real valued $f$ is given by its Fourier expansion 
$$
f(x)=\sum_{n\in \Z}\hat{f}(n)e^{inx},\quad \hat{f}(n)=
\overline{\hat{f}(-n)},
$$
for $s\in\R$, we define its Sobolev norm as
\begin{equation}\label{norm}
\|f\|_{{\mathcal H}^s}=\|f\|_{s}=\Big(\sum_{n\in \Z}\langle n\rangle^{2s}
|\hat{f}(n)|^2\Big)^{\frac{1}{2}}\,\, ,
\end{equation}
where $\langle n\rangle =1+|n|$.
We denote by ${\mathcal H}^s$ the space of real valued 
$f$ such that \eqref{norm} is finite. It is well known that ${\mathcal H}^s$ is a Hilbert space (with the natural scalar product).
We denote by $H^s$ the closed subspace of  ${\mathcal H}^s$ of functions with zero Fourier coefficient , i.e.
$$
H^s=\{f\in {\mathcal H}^s\,:\, \hat{f}(0)=0\}\,.
$$
Since the mean value is preserved by \eqref{BBM-gamma}, we have that $H^s$ is a natural space for the solutions of \eqref{BBM-gamma}.
We consider $H^s$, equipped with the norm
\begin{equation}\label{norm_bis}
\|f\|_{H^s}=\frac{1}{\sqrt{2}}\Big(\sum_{n\in \Z}|n|^{2s}
|\hat{f}(n)|^2\Big)^{\frac{1}{2}}
=
\Big(\sum_{n=1}^{\infty}|n|^{2s}
|\hat{f}(n)|^2\Big)^{\frac{1}{2}}
\,\, .
\end{equation}
For elements in $H^s$ the norms \eqref{norm} and \eqref{norm_bis} are equivalent.
We define the Fourier multipliers $|D_x|^s$ as 
$$
|D_x|^s(f)(x)=\sum_{n\in \Z}|n|^s\hat{f}(n)e^{inx}\,.
$$
Then 
$$
\|f\|_{H^s}=\frac{1}{2\sqrt{\pi}}\||D_x|^s f\|_{L^2}
$$
We denote by $\pi_N$ the Dirichlet projector, i.e.
$$
\pi_{N}(f)=\sum_{|n|\leq N}\hat{f}(n)e^{inx}\,.
$$
We denote by $\mu_{s,r}$, the measure defined by
$$
d\mu_{s,r}(u)=\chi_{r}(u)d\mu_{s}(u),
$$
where 
$$
\chi_r(u)=\chi\Big(r^{-1}
\big(
\|u\|_{L^2}^2+4\pi\|u\|^2_{H^{\frac{\gamma}{2}}}
\big)
\Big)
$$
and 
$\chi : \R\rightarrow\R$ denotes the characteristic function of the set $[0,1]$.

For $s\in\R$ and $R\geq 0$, we set $B_{R,s}=\{u\in H^s\,:\, \|u\|_{H^s}\leq R \}$.
\section{Construction and general properties of the flows and the approximated flows}
\subsection{Existence of the dynamics}
Consider the truncated version of \eqref{BBM-gamma}
\begin{equation}\label{BBM-gamma-N}
\partial_t u+\partial_t|D_x|^\gamma u+\partial_x u+\partial_x\pi_{N}((\pi_{N}u)^2)=0\,.
\end{equation}
We consider \eqref{BBM-gamma-N}, posed on the one dimensional torus and with initial data in $H^s$. 
\begin{lemme}\label{es}
Let $\gamma>1$. Then for every $\sigma\geq 0$,
$$
\|(1+|D_x|^{\gamma})^{-1}\partial_x(uv)\|_{\sigma}\leq C_{\sigma}\big( \|u\|_{\sigma}\|v\|_{\gamma/2}+\|u\|_{\gamma/2}\|v\|_{\sigma}\big)\,.
$$
\end{lemme}
\begin{proof}
Since $(1+|D_x|^{\gamma})^{-1}\partial_x$ is bounded on ${\mathcal H}^\sigma$,
the proof of direct consequence of the classical product estimate
\begin{equation}\label{AG}
\|uv\|_{\sigma}\leq 
C_{\sigma}\big(\|u\|_{\sigma}\|v\|_{L^\infty}+\|u\|_{L^\infty}\|v\|_{\sigma}\big)
\end{equation}
and the Sobolev embedding ${\mathcal H}^{\gamma/2}\subset L^\infty$. 
\end{proof}
Using Lemma~\ref{es} one gets the following uniform in $N$ local well-posedness result for \eqref{BBM-gamma-N}.
\begin{lemme}\label{local}
Let $\sigma\geq \gamma/2$. Then for every $u(0)\in H^\sigma$ there is a time $\tau>0$ depending only on $\|u(0)\|_{H^{\gamma/2}}$ and a unique solution of 
\eqref{BBM-gamma-N} in $C([-\tau,\tau];H^\sigma)$ with initial data $u(0)$. Moreover $\|u\|_{L^\infty([-\tau,\tau];H^\sigma)}\leq 2\|u(0)\|_{H^\sigma}$.
\end{lemme}
\begin{proof}
We proceed by a fixed point argument. 
One may rewrite \eqref{BBM-gamma-N}  as the integral equation 
\begin{equation}\label{fp}
u(t)=S(t)(u(0))-\int_{0}^{t}S(t-\tau)\big((1+|D_x|^{\gamma})^{-1}\partial_x\pi_{N}((\pi_{N}u(\tau))^2)\big)d\tau.
\end{equation}
Using Lemma~\ref{es}, one can look for a fixed point of the map
$
F_{u(0)}(u),
$
defined by the right hand-side of \eqref{fp} in a suitable ball of the space 
$C([-\tau,\tau];H^{\gamma/2})$, where $\tau=c(1+\|u(0)\|_{H^{\gamma/2}})^{-1}$ and $c$ is a small constant. 
Indeed, using Lemma~\ref{es} and the uniform bounds for $\pi_N$ on $H^{\gamma/2}$, we can obtain that for $u\in C([-\tau,\tau];H^{\gamma/2})$,
$$
\|F_{u(0)}(u)\|_{L^\infty([-\tau,\tau];H^{\gamma/2})}
\leq \|u_0\|_{H^{\gamma/2}}+C\tau
\|u\|^2_{L^\infty([-\tau,\tau];H^{\gamma/2})}
\,.
$$
Therefore, the space $E$ defined by
$$
E\equiv \{u\in C([-\tau,\tau];H^{\gamma/2})\,:\, \|u\|_{L^\infty([-\tau,\tau];H^{\gamma/2})}\leq 2\|u(0)\|_{H^{\gamma/2}}\}
$$
is such that $F_{u(0)}(E)\subset E$, provided the constant $c$ in the definition of $\tau$ is small enough.
By invoking once again Lemma~\ref{es} and the uniform bounds for $\pi_N$ on $H^{\gamma/2}$, we can obtain that for $u,v\in E$,
$$
\|F_{u(0)}(u)-F_{u(0)}(v)\|_{L^\infty([-\tau,\tau];H^{\gamma/2})}
\leq \frac{1}{2}\|u-v\|_{L^\infty([-\tau,\tau];H^{\gamma/2})}\,,
$$
by possibly taking an even smaller value of the constant $c$ involved in the definition of $\tau$. Therefore 
$F_{u(0)}$ is a contraction on $E$. The fixed point of this contraction provides the solution of \eqref{BBM-gamma-N} we look for.

Let us now turn the propagation the $H^\sigma$ regularity of the obtained solution $u$. This regularity is preserved for very small times of order $(1+\|u(0)\|_{H^\sigma})^{-1}$ by the fixed point argument that we have just presented. In order to show that the regularity is preserved for longer times we use  Lemma~\ref{es} and the uniform bounds for $\pi_N$ on $H^\sigma$, in order to obtain that
$$
\|u\|_{L^\infty([-\tau,\tau];H^\sigma)}\leq \|u_0\|_{H^\sigma}+C_{\sigma}\tau\
\|u\|_{L^\infty([-\tau,\tau];H^\sigma)}\|u\|_{L^\infty([-\tau,\tau];H^{\gamma/2})}\,.
$$
Therefore the $H^\sigma$ regularity is preserved for time $\tau=c(1+\|u(0)\|_{H^{\gamma/2}})^{-1}$, where the constant $c$ is sufficiently small depending only on 
$\sigma$.

The uniqueness statement follows by using that if $u_1$ and $u_2$ are two solutions of  \eqref{BBM-gamma-N}  in 
$C([-\tau,\tau];H^\sigma)$ then by using  Lemma~\ref{es} and the uniform bounds for $\pi_N$ on $H^\sigma$, we can obtain that for any interval $[-\tau_1,\tau_1]$,
$\tau_1\leq\tau$,
$$
\|u_1-u_2\|_{L^\infty([-\tau_1,\tau_1];H^\sigma)}
\leq C_{\sigma}\tau_1 \|u_1-u_2\|_{L^\infty([-\tau_1,\tau_1];H^\sigma)}
\|u_1+u_2\|_{L^\infty([-\tau,\tau];H^\sigma)}\,.
$$
We therefore conclude that $u_1=u_2$ on $[-\tau_1,\tau_1]$, where $\tau_1$ is such that
$$
\tau_1 C_{\sigma}\big( \|u_1\|_{L^\infty([-\tau,\tau];H^\sigma)}
+ \|u_2\|_{L^\infty([-\tau,\tau];H^\sigma)}\big)<\frac{1}{2}\,.
$$
Then we cover $[-\tau,\tau]$ by intervals of size $\tau_1$ and we repeat the previous reasoning to conclude that $u_1=u_2$ on $[-\tau,\tau]$.
The continuity statements are consequences of the previous analysis. 
This completes the proof of Lemma~\ref{local}.
\end{proof}
By the invariance of \eqref{BBM-gamma} 
with respect to time translations, the statement of Lemma~\ref{local} with initial data given at any time $t_0\in \R$ on the interval $[t_0-\tau,t_0+\tau]$, where $\tau$ 
depends only of $\|u(t_0)\|_{H^{\gamma/2}}$. 

One can similarly obtain the following local well-posedness result for \eqref{BBM-gamma}.
\begin{lemme}\label{local_bis}
Let $\sigma\geq \gamma/2$. Then for every $u(0)\in H^\sigma$ there is a time $\tau>0$ depending only on $\|u(0)\|_{H^{\gamma/2}}$ and a unique solution of \eqref{BBM-gamma} in $C([-\tau,\tau];H^\sigma)$ with initial data $u(0)$. Moreover 
\begin{equation}\label{local_bound}
\|u\|_{L^\infty([-\tau,\tau];H^\sigma)}\leq 2\|u(0)\|_{H^\sigma}
\end{equation}
\end{lemme}
The next lemma is of key importance.
\begin{lemme}\label{boris}
Let $u$ be a local solution of \eqref{BBM-gamma-N}, given by Proposition~\ref{local}. 
Then 
\begin{equation}\label{CL}
\frac{d}{dt}
\Big(
\|u(t)\|_{L^2}^2+4\pi\|u(t)\|_{H^{\gamma/2}}^2\Big)=0.
\end{equation}
A similar statement holds for the solutions of \eqref{BBM-gamma}.
\end{lemme}
\begin{proof}
Let $u$ be a local solution of \eqref{BBM-gamma-N}, given by Proposition~\ref{local}.  Then we take the $L^2$ scalar product of
\eqref{BBM-gamma-N}  with $u$ and  using that
$$
(\partial_x\pi_{N}((\pi_{N}u)^2),u)=\frac{1}{3}\int \partial_x((\pi_{N}u)^3)=0,\quad (u,\partial_x u)=0,
$$ 
and
$$
(\partial_t u,u)=\frac{1}{2}\partial_t\|u\|_{L^2}^2,\quad
(\partial_t |D_x|^\gamma u,u)=\frac{1}{2}\partial_t\||D_x|^{\gamma/2}u\|_{L^2}^2\,,
$$
we obtain that  \eqref{CL} holds.
\end{proof}
Using Lemma~\ref{local}, Lemma~\ref{local_bis} and the conservation law displayed by Lemma~\ref{boris}  
one readily gets Proposition~\ref{l1} and also the global well-posedness in $H^\sigma$ of \eqref{BBM-gamma-N}, "uniformly" in $N$.
Let us denote by $\Phi(t)$ and $\Phi_N(t)$ the global flows on $H^{\sigma}$, $\sigma\geq \gamma/2$ of \eqref{BBM-gamma} and \eqref{BBM-gamma-N} respectively.
By iterating the local bounds we get the following statement.
\begin{proposition}\label{exp}
Let $\sigma\geq \gamma/2$.  For every every $R>0$ there is a constant $C$ such that
for every $v\in H^\sigma$ such that $\|v\|_{H^{\gamma/2}}\leq R$ and every $N\geq 1$, one has the the bound
$$
\|\Phi(t)(v)\|_{H^\sigma}+\|\Phi_{N}(t)(v)\|_{H^\sigma}
\leq e^{C(1+|t|)}\|v\|_{H^\sigma}\,.
$$
\end{proposition}
\begin{proof}
Let $u$ be a solution of \eqref{BBM-gamma-N}.
Using \eqref{CL} for every $t$ we can iterate the local bound \eqref{local_bound}  $[|t|/\tau]+1$ times to obtain that
$$
\|u(t)\|_{H^{\sigma}}\leq 2^{[|t|/\tau]+1}\|u(0)\|_{H^\sigma}\,.
$$
A similar analysis applies for the solutions of \eqref{BBM-gamma}.
This completes the proof of Proposition~\ref{exp}.
\end{proof}
As a consequence of Proposition~\ref{exp}, one also gets the following statement. 
\begin{proposition}\label{union}
Let $\sigma\geq \gamma/2$. 
Then for every $T>0$ and $R>0$ there exists $R'>0$ such that 
$$
\bigcup_{N\in\N}\,\bigcup_{t\in[-T,T]}\Big(\Phi_N(t)(B_{R,\sigma})\cup \Phi(t)(B_{R,\sigma})\Big)\subset B_{R',\sigma}\,.
$$
\end{proposition}
\subsection{Approximation properties}
We have the following basic approximation property.
\begin{proposition}\label{standard}
Let $\sigma\geq \gamma/2$. Fix $t\in\R$, $R>0$ and a compact $K\subset B_{R,\sigma}$. 
Then for every $\varepsilon>0$ there exists $N_0$ such that for every $N\geq N_0$, 
$$
\|\Phi(t)(v)-\Phi_{N}(t)(v)\|_{H^\sigma}<\varepsilon \, ,\quad \forall\, v\in K\,.
$$
\end{proposition}
\begin{proof}
Let $u$ and $u_N$, be solutions of  \eqref{BBM-gamma} and \eqref{BBM-gamma-N} with the same initial data $v\in K$.
Then $w_N\equiv u-\pi_N u_N$ solves the equation 
\begin{equation}\label{difference}
\partial_t w_N+\partial_t|D_x|^\gamma w_N+\partial_x w_N+
\partial_x(u^2-\pi_{N}((\pi_{N}u)^2))=0\,.
\end{equation}
Next, we can write
$$
u^2-\pi_{N}((\pi_{N}u_N)^2)=(1-\pi_{N})(u^2)+\pi_{N}\big(w_{N}(u+\pi_{N}u_N)\big)\,.
$$
By using the estimate of Lemma~\ref{es} and Proposition~\ref{union}, we obtain that there is a constant $C$ only depending on $\sigma$, $t$ and $R$ such that
\begin{equation}\label{unif}
\|w_N\|_{L^\infty([-\tau,\tau];H^\sigma)}
\leq  C\tau \|w_N\|_{L^\infty([-\tau,\tau];H^\sigma)}
+C\|(1-\pi_{N})(u^2)\|_{L^\infty([-\tau,\tau];H^\sigma)}\,.
\end{equation}
We are now in position to use the following lemma.
\begin{lemme}\label{compact}
For every $\varepsilon>0$ there is $N_0$ such that for every $N\geq N_0$ and every $v\in K$,
$$
\|(1-\pi_{N})((\Phi(t)(v))^2)\|_{L^\infty([-\tau,\tau];H^\sigma)}<\varepsilon.
$$
\end{lemme}
\begin{proof}
Let $\varepsilon>0$. Since the map $v\mapsto \Phi(t)(v)$ is continuous from $H^\sigma$ to $C([-\tau,\tau];H^\sigma)$, we obtain that the image of $K$ under this map is a compact in $C([-\tau,\tau];H^\sigma)$. 
Therefore, using   Lemma~\ref{es} and Proposition~\ref{union}, we obtain that there exists a finite set $J$ such that $(\Phi(t)(v_j))_{j\in J}$ has the propety,
$$
\forall\, v\in K,\quad\exists\, j\in J,\,\,\,\, 
\|(\Phi(t)(v))^2-(\Phi(t)(v_j))^2\|_{L^\infty([-\tau,\tau];H^\sigma)}<\frac{\varepsilon}{2}\,.
$$
Therefore, it remains to show that for every $j$ in the finite set $J$
there is $N_0$ such that for every $N\geq N_0$,
$$
\|(1-\pi_{N})((\Phi(t)(v_j))^2)\|_{L^\infty([-\tau,\tau];H^\sigma)}<\frac{\varepsilon}{2}.
$$
Now, thanks to the (uniform) continuity of the map $t\mapsto (\Phi(t)(v_j))^2$ from $[-\tau,\tau]$ to $H^\sigma$, we obtain that there is a finite set
$(t_l)_{l\in\Lambda}$ of $[-\tau,\tau]$ such that
$$
\forall\, t\in [-\tau,\tau], \quad\exists\, l\in\Lambda,\quad
\|(\Phi(t)(v_j))^2-(\Phi(t_l)(v_j))^2\|_{H^\sigma}<\frac{\varepsilon}{4}\,.
$$
We are therefore reduced to showing that for every $j\in J$ and every $l\in\Lambda$ there is $N_0$ such that for every $N\geq N_0$,
$$
\|(1-\pi_{N})((\Phi(t_l)(v_j))^2)\|_{H^\sigma}<\frac{\varepsilon}{4}.
$$
The last statement is a direct consequence of the definition of the Sobolev space $H^\sigma$.
This completes the proof of Lemma~\ref{compact}.
\end{proof}
Using Lemma~\ref{compact}, we obtain that if the constant $C$ involved in \eqref{unif} satisfies $C\tau<\frac{1}{2}$ then
$$
\|w_N\|_{L^\infty([-\tau,\tau];H^\sigma)}
\leq  \Lambda(N,v),
$$
where here and the sequel of the proof we denote by $\Lambda(N,v)$ a generic 
quantity such that for every $\varepsilon>0$ there is $N_0$ such that for every $N\geq N_0$ and every $v\in K$,
$
\Lambda(N,v)<\varepsilon.
$

Next, we can perform the same analysis to arrive at the bound 
$$
\|w_N\|_{L^\infty([\tau,2\tau];H^\sigma)}
\leq \|u(\tau)-\pi_Nu_{N}(\tau)\|_{H^\sigma}+
C\tau \|w_N\|_{L^\infty([-\tau,\tau];H^\sigma)}+\Lambda(N,v),
$$
which implies 
$$
\|w_N\|_{L^\infty([\tau,2\tau];H^\sigma)} \leq  \Lambda(N,v)\,.
$$
Now we can cover the interval $[-t,t]$ by intervals of size $\tau$ and repeat the previous analysis to arrive at the bound 
\begin{equation}\label{I}
\|u(t)-\pi_Nu_{N}(t)\|_{H^\sigma}\leq \Lambda(N,v),
\end{equation}
Next, we write
\begin{equation}\label{CS}
u-u_N=u-\pi_Nu_N-(1-\pi_N)u_N\,.
\end{equation}
We now invoke the following lemma.
\begin{lemme}\label{compact_bis}
Fix $t\in\R$.
For every $\varepsilon>0$ there is $N_0$ such that for every $N\geq N_0$ and every $v\in K$,
$$
\|(1-\pi_{N})(\Phi_N(t)(v))\|_{H^\sigma)}<\varepsilon.
$$
\end{lemme}
The proof of Lemma~\ref{compact_bis} is similar (simpler) to the proof of Lemma~\ref{compact} and therefore will be omitted.
We now come back to \eqref{CS} and we use \eqref{I} and Lemma~\ref{compact_bis}.
This completes the proof of Proposition~\ref{standard}.
\end{proof}
As a consequence of Proposition~\ref{standard}, we also have the following approximation property. 
\begin{proposition}\label{approximation}
Let $\sigma\geq \gamma/2$. Fix $t\in\R$, $R>0$ and a compact $A\subset B_{R,\sigma}$. 
For every $\varepsilon>0$ there exists $N_0$ such that for every $N\geq N_0$,  
$$
\Phi(t)(A)\subset \Phi_{N}(t)(A+B_{\varepsilon,\sigma}).
$$
\end{proposition}
\begin{proof}
Let $u\in \Phi(t)(A)$. This means that there exists $v\in A$ such that $u=\Phi(t)(v)$. Write
$$
u=\Phi_{N}(t)\Big(\Phi_{N}(-t)\Phi(t)(v)\Big)\,.
$$
Set $w_N=\Phi_{N}(-t)\Phi(t)(v)$. The goal is to show that 
$
w_N\in A+B_{\varepsilon,\sigma}
$
for every $N\geq N_{0}(\varepsilon,t, A)$. 
For that purpose, we can write
$$
w_N=v+z_N,\quad z_{N}\equiv \Phi_{N}(-t)\Phi(t)(v)-v\,.
$$
Since $v\in A$, the issue is to check that $z_N\in B_{\varepsilon,\sigma}
$
for every $N\geq N_{0}(\varepsilon,t,A)$. 
We can write
$$
z_{N}=\Phi_{N}(-t)(\Phi(t)(v)-\Phi_N(t)(v))\,.
$$
Using Proposition~\ref{exp}, we obtain that
\begin{equation}\label{parvo}
\|\Phi_{N}(-t)(\Phi(t)(v)-\Phi_N(t)(v))\|_{H^\sigma}
\leq C(t,R)\|\Phi(t)(v)-\Phi_N(t)(v)\|_{H^\sigma}\,.
\end{equation}
Using Proposition~\ref{standard}, we obtain that
\begin{equation}\label{vtoro}
\|\Phi(t)(v)-\Phi_N(t)(v)\|_{H^\sigma}\leq \Lambda(N,v),
\end{equation}
where 
for every $\varepsilon>0$ there is $N_0$ such that for every $N\geq N_0$ and every $v\in A$,
$
\Lambda(N,v)<\varepsilon.
$
A combination of \eqref{parvo} and \eqref{vtoro} implies that
$\|z_{N}\|_{H^\sigma}<\varepsilon,$ provided $N\geq N_0(\varepsilon, t,A)$.
This completes the proof of Proposition~\ref{approximation}.
\end{proof}
\section{The case $\gamma>2$ as a consequence of Ramer's result}\label{Ramm}
We will show in this section that in the case $\gamma>2$ the result of Theorem~\ref{main}  follows from \cite{Ra}.
Thanks to the Duhamel formula, we can write 
$$
\Phi(t)=S(t)\circ\Psi(t),
$$ 
where
\begin{equation}\label{id}
\Psi(t)(u_0)=u_0-\int_{0}^{t}S(-\tau)\big((1+|D_x|^{\gamma})^{-1}\partial_x((\Phi(\tau)(u_0))^2)\big)d\tau.
\end{equation}
Thanks to the invariance of the complex gaussians under rotations, the measure $\mu_s$ is invariant under $S(t)$ (see Lemma~\ref{rotation} bellow).
Therefore, we need to show the quasi-invariance of $\mu_s$ under $\Psi(t)$.
Take $\sigma_1>1/2$ to be chosen later. Write $\Psi(t)={\rm Id}+K(t)$, where $K(t)$ is defined via \eqref{id}.
Thanks to \cite{Ra} and the analysis of the previous section, 
the measure $\mu_s$ is quasi-invariant under $\Psi(t)$, if we can show that for $u_0$ in a bounded set of $H^{s+\frac{\gamma}{2}-\sigma_1}$ and $|t|\leq 1$ small enough (depending only on the fixed bounded set) we have that the map $(DK(t))_{u_0}$ is a Hilbert-Schmidt map on $H^{s+\frac{\gamma}{2}}$.
One now can compute and arrive at the expression
$$
(DK(t))_{u_0}(v_0)=-2\int_{0}^{t}S(-\tau)\big((1+|D_x|^{\gamma})^{-1}\partial_x(\Phi(\tau)(u_0)v(\tau)\big)d\tau,
$$
where $v$ is a solution of the linear problem
\begin{equation}\label{pb_v}
\partial_t v+\partial_t|D_x|^\gamma v+\partial_x v+2\partial_x(\Phi(t)(u_0)v)=0,\quad
v|_{t=0}=v_0\,.
\end{equation}
Next, for $\sigma_2>1/2$ to be chosen later, we write
$$
(DK(t))_{u_0}=(1+|D_x|)^{-\sigma_2}\circ A,
$$
where
$$
A(v_0)\equiv-2\int_{0}^{t}S(-\tau)\big((1+|D_x|)^{\sigma_2}
(1+|D_x|^{\gamma})^{-1}
\partial_x(\Phi(\tau)(u_0)v(\tau)\big)d\tau.
$$
Since $(1+|D_x|)^{-\sigma_2}$ is a Hilbert-Schmidt map on  $H^{s+\frac{\gamma}{2}}$ and the later property is preserved by compositions with bounded maps, we are reduced to show that the map $A$ is bounded on  $H^{s+\frac{\gamma}{2}}$. The assumption $\gamma>2$ will be used in the verification of this property.
Using \eqref{AG}, we can write
\begin{eqnarray*}
\|A(v_0)\|_{H^{s+\frac{\gamma}{2}}} & \leq & C \sup_{\tau\in[0,t]}\|\Phi(\tau)(u_0)v(\tau)\|_{H^{s-\frac{\gamma}{2}+\sigma_2+1}}
\\
& \leq &
C \sup_{\tau\in[0,t]}
\Big(\|\Phi(\tau)(u_0)\|_{H^{s-\frac{\gamma}{2}+\sigma_2+1}}\|v(\tau)\|_{H^{s-\frac{\gamma}{2}+\sigma_2+1}}\Big)\,.
\end{eqnarray*}
We now estimate each of the factors. Thanks to the assumption $\gamma>2$ for $\sigma_1$ and $\sigma_2$ close enough to $1/2$, we have
$$
s-\frac{\gamma}{2}+\sigma_2+1<s+\frac{\gamma}{2}-\sigma_1\,.
$$
Therefore, using the results of the previous section, we obtain that for $\tau\in[0,t]$,
$$
\|\Phi(\tau)(u_0)\|_{H^{s-\frac{\gamma}{2}+\sigma_2+1}}\leq\|\Phi(\tau)(u_0)\|_{H^{s+\frac{\gamma}{2}-\sigma_1}}
\leq C\|u_0\|_{H^{s+\frac{\gamma}{2}-\sigma_1}}\,.
$$
Next, coming back to \eqref{pb_v} and using \eqref{AG}, we get that for $\tau\in [0,t]$,
$$
\|v(\tau)\|_{H^{s-\frac{\gamma}{2}+\sigma_2+1}}
\leq 
\|v_0\|_{H^{s-\frac{\gamma}{2}+\sigma_2+1}}
+
C|t| \sup_{\tau\in[0,t]}\|\Phi(\tau)(u_0)v(\tau)\|_{H^{s-\frac{\gamma}{2}+\sigma_2+1}}
$$
(here we only use that $\gamma\geq 1$). 
Therefore using that for $\sigma_2$ close enough to $1/2$, $s-\frac{\gamma}{2}+\sigma_2+1<s+\frac{\gamma}{2}$, we obtain 
$$
\|v(\tau)\|_{H^{s-\frac{\gamma}{2}+\sigma_2+1}}
\leq 
\|v_0\|_{H^{s+\frac{\gamma}{2}}}
+
C|t|\|u_0\|_{H^{s+\frac{\gamma}{2}-\sigma_1}}
 \sup_{\tau\in[0,t]}\|v(\tau)\|_{H^{s-\frac{\gamma}{2}+\sigma_2+1}}
$$
Hence that for $t$ small enough, but still only depending on the bounded set of $H^{s+\frac{\gamma}{2}-\sigma_1}$ where $u_0$ ranges, one has the bound
$$ 
\|v(\tau)\|_{H^{s-\frac{\gamma}{2}+\sigma_2+1}}
\leq 
2\|v_0\|_{H^{s+\frac{\gamma}{2}}},\quad \tau\in[0,t]\,.
$$
Therefore, we obtain that the map $A$ is bounded on $H^{s+\frac{\gamma}{2}}$. This in tun implies that $(DK(t))_{u_0}$ is a Hilbert-Schmidt map and consequently we can apply the result of \cite{Ra} to get the result of Theorem~\ref{main} for $\gamma>2$.

From now on we shall suppose that  $\gamma\in(4/3,2]$. This is the range of $\gamma$ which does not seem covered by \cite{Ra}. 
In this region of $\gamma$, we shall use more involved properties of the transformation $\Phi(t)$.
These considerations seem to go beyond the analysis of general maps close to the identity done in \cite{Ra}. 
Let us also mention that
the result of \cite{Ra} under Hilbert-Schmidt assumption is already a quite non trivial result using a stochastic interpretation of the obtained densities (the straightforward result being obtained under a trace class assumption). 
\section{A variable change formula}
For every $N$, we denote by $E_N$ the real vector space spanned by 
$$
(\cos(nx),\sin(nx))_{1\leq n\leq N}\,.
$$
We equip $E_N$ with the natural scalar product.
We endow $E_N$ with a Lebesgue measure $L_N$ as follows.
If 
$$
(\pi_Nu)(x)=\sum_{0<|n|\leq N} u_n\, e^{inx},\quad u_n=\overline{u_{-n}}
$$
and $u_n=a_n+ib_n$, $(a_n,b_n)\in \R^2$ then
$$
(\pi_Nu)(x)=\sum_{n=1}^{N}(a_n (2\cos(nx))+b_n (-2\sin(nx)))\,.
$$
Therefore, we denote by $L_N$ the Lebesgue measure on $E_N$ build with respect to the orthogonal basis 
$$
(2\cos(nx),-2\sin(nx))_{1\leq n\leq N}\,.
$$
Next, we denote by $E_N^\perp$ the orthogonal complement of $E_N$ in $H^s$. 
We endow $E_N^\perp$ with the measure $\mu^\perp_{s;N}$ which is the image measure under the map
$$
\omega\longmapsto \sum_{|n|>N}\frac{g_n(\omega)}{|n|^{s+\gamma/2}}e^{inx}\,.
$$
We can now see the measure  $\mu_s$ as a product measure on $E_N\times E_N^\perp$ as follows
\begin{eqnarray*}
d\mu_s & = & 
\gamma_N e^{-\sum_{n=1}^{N}|n|^{2s+\gamma}(a_n^2+b_n^2)}
dL_N(a_1,b_1,\dots,a_N,b_N) \, d\mu^\perp_{s;N}
\\
& = &
\gamma_N e^{-\|\pi_N u\|_{H^{s+\frac{\gamma}{2}}}^2}\,\, du_1...du_N \, d\mu^\perp_{s;N}
\end{eqnarray*}
where $\gamma_N$ is a suitable renormalization factor and
$$
(\pi_Nu)(x)=\sum_{0<|n|\leq N} u_n\, e^{inx},\quad u_n=\overline{u_{-n}},\quad u_n=a_n+ib_n,\, (a_n,b_n)\in\R^2\,.
$$
We have the following "change of variables rule".
\begin{proposition}\label{ch-var} 
For $A$ a Borel set of $H^s$ one has the identity
\begin{eqnarray*}
\mu_{s,r}(\Phi_N(t)(A)) & = &\int_{\Phi_N(t)(A)} \chi_{r}(u) d\mu_s(u)
\\
& = &
\gamma_{N}\int_{A} 
\chi_{r}(u) 
e^{-\|\pi_N(\Phi_N(t)(u))\|^2_{H^{s+\frac{\gamma}{2}}}} du_1...du_N\, d\mu^\perp_{s;N}
  \end{eqnarray*}
\end{proposition}
\begin{proof}
We follow \cite{TV}.  A difference with \cite{TV} is that in Proposition~\ref{ch-var}, we deal with $\chi_{r}(u)$ and not $\chi_{r}(\pi_N u)$. As we will see below thanks to the conservation law of Lemma~\ref{boris} the analysis is not affected by the lack of the projector $\pi_N$.
Let us denote by $\tilde{\Phi}_N(t)$ the (well-defined) flow of the following ODE on $E_N$,
\begin{equation}\label{FD}
\partial_t u+\partial_t|D_x|^\gamma u+\partial_x  u+\partial_x\pi_{N}(u^2)=0,\quad u(0,x)\in E_N\,.
\end{equation}
Then, by definition, we have the following relation 
\begin{equation}\label{split}
\Phi_N(t)(u_0)=\tilde{\Phi}_{N}(t)(\pi_N u_0)+S(t)((1-\pi_{N})u_0)\,.
\end{equation}
We have the following lemma.
\begin{lemme}\label{liouvile}
The measure $ du_1...du_N $ is invariant under the flow $\tilde{\Phi}_N(t)$. 
\end{lemme}
\begin{proof}
If 
$$
u(x)=\sum_{0<|n|\leq N} u_n\, e^{inx},\quad u_n=\overline{u_{-n}}
$$
then the equation \eqref{FD} can be rewritten as 
$$
\partial_t u_n=-\frac{in}{1+|n|^\gamma}\Big(u_n+\sum_{\substack{n_1+n_2=n\\0<|n_1|,|n_2|\leq N}} u_{n_1}u_{n_2}\Big),\quad 1\leq n\leq N
$$
and thus if $u_n=a_n+ib_n$, we arrive at the equations
\begin{eqnarray*}
\partial_t a_n & = &\frac{n}{1+|n|^\gamma}\Big(b_n+\sum_{\substack{n_1+n_2=n\\0<|n_1|,|n_2|\leq N}} (a_{n_1}b_{n_2}+a_{n_2}b_{n_1})\Big),
\\
\partial_t b_n & = &-\frac{n}{1+|n|^\gamma}\Big(a_n+\sum_{\substack{n_1+n_2=n\\0<|n_1|,|n_2|\leq N}} (a_{n_1}a_{n_2}-b_{n_1}b_{n_2})\Big)\,.
\end{eqnarray*}
We now observe that if we write the last equations as
$$
\partial_{t}a_n=F_{n}(a_1,\dots,a_N,b_1\dots,b_N),\quad
\partial_{t}b_n=G_{n}(a_1,\dots,a_N,b_1\dots,b_N),
$$
then we have the remarkable property
$$
\frac{\partial F_n}{\partial a_n}=\frac{\partial G_n}{\partial b_n}=0,\quad 1\leq n\leq N\,.
$$
In particular, the ODE \eqref{FD} is generated by a divergence free vector field. 
Therefore the statement of Lemma~\ref{liouvile} follows from the Liouville theorem.
\end{proof}
We also have the following statement.
\begin{lemme}\label{rotation}
The measure $\mu_{s;N}^\perp$ on 
$E_N^\perp$ is invariant under the map $S(t)$. In particular $\mu_s$ is invariant under $S(t)$. 
\end{lemme} 
\begin{proof}
The proof of Lemma~\ref{rotation} follows from the invariance of the complex gaussian under rotations. 
We follow closely  \cite[Lemma~5.3]{TV}, where the proof of an analogous statement is given. 
For $M>N$, we denote by $E_N^M$ the finite dimensional real vector space spanned by $(\cos(nx),\sin(nx))$, where $N<n\leq M$.
We denote by $\mu_N^M$ the centered gaussian measure on $E_N^M$ induced by the series 
$$
\sum_{|n|=N+1}^M \frac{g_n(\omega)}{|n|^{s+\gamma/2}} e^{inx}\,.
$$
By using the Fatou lemma, we obtain that if $U$ is an open set of $E_N$ then we have
\begin{equation}\label{fatou}
\mu^\perp_N(U)\leq \liminf_{M\rightarrow\infty}\mu_N^M(U\cap E^M_N)\,.
\end{equation}
By passing to a complementary set in \eqref{fatou}, we get that for $F$ a closed set of $E_N$,
\begin{equation}\label{fatou-bis}
\mu^\perp_N(F)\geq \limsup_{M\rightarrow\infty}\mu_N^M(F\cap E^M_N)\,.
\end{equation}
By the definition of $S(t)$, we get
\begin{eqnarray*}
S(t)(\cos(nx)) &  =  & \cos\Big(-\frac{tn}{1+|n|^\gamma}+nx\Big),
\\
S(t)(\sin(nx))   & =  &  \sin\Big(-\frac{tn}{1+|n|^\gamma}+nx\Big).
\end{eqnarray*}
Therefore for fixed $t$, the map $S(t)$ acts as a rotation on $E_N^M$.
Consequently,  by the invariance of centered gaussians by rotations, we obtain that 
the measure $\mu_N^M$ is invariant under $S(t)$. Let $F$ be a closed set of $E_N^\perp$.  Then $S(t)(F)$ is also closed and thanks to \eqref{fatou-bis},
$$
\mu^\perp_N(S(t)(F)+\overline{B_{\varepsilon}})\geq \limsup_{M\rightarrow\infty}\mu_N^M((S(t)F+\overline{B_{\varepsilon}})\cap E^M_N),
$$
where $B_{\varepsilon}$ denotes the open ball of radius $\varepsilon$ in $E_N^\perp$ ($E_N^\perp$ is equipped with the $H^s$ topology). 
Using that $S(t)$ acts as an isometry on $H^s$ and the invariance of $E^M_N$ under $S(t)$,  we obtain that for every $\varepsilon$ and every $M$,
$$
S(t)\big((F+B_{\varepsilon})\cap E^M_N\big)\subset (S(t)F+\overline{B_{\varepsilon}})\cap E^M_N.
$$
Therefore using the invariance of $\mu_N^M$ under $S(t)$ and \eqref{fatou}, we get
\begin{eqnarray*}
\mu^\perp_N(S(t)(F)+\overline{B_{\varepsilon}}) &\geq & \limsup _{M\rightarrow\infty}\mu_N^M\big(S(t)\big((F+B_{\varepsilon})\cap E^M_N\big)\big)
\\
& = &
\limsup _{M\rightarrow\infty}\mu_N^M\big((F+B_{\varepsilon})\cap E^M_N\big)
\\
& \geq &
\mu_N^\perp(F+B_{\varepsilon})\geq \mu_N^\perp(F)\,.
\end{eqnarray*}
Letting $\varepsilon\rightarrow 0$ and using the Lebesgue theorem we get $\mu_N^\perp(F)\leq \mu_N^\perp(S(t)(F))$.
By the time reversibility of $S(t)$, we get $\mu_N^\perp(F)= \mu_N^\perp(S(t)(F))$ for every closed set $F$ of $E^N$.
Finally by approximation arguments, we obtain that
$\mu_N^\perp(A)= \mu_N^\perp(S(t)(A))$ for every Borel set $A$ of $E^N$. This completes the proof of Lemma~\ref{rotation}.
\end{proof}
Let us now complete the proof of Proposition~\ref{ch-var}. Again, we follow closely \cite{TV}.
Recall that $dL_N=du_1...du_N$.  We can write
\begin{equation}\label{expression}
\int_{\Phi_N(t)(A)} 
\chi_r(u)
e^{-\|\pi_N u\|_{H^{s+\frac{\gamma}{2}}}^2}
dL_N \,d\mu^\perp_{s;N}
\end{equation}
as
$$
\int_{E_N}\int_{E_N^\perp}\11(\Phi_N(t)(A))(u)\chi_r(u)
e^{-\|\pi_N u\|_{H^{s+\frac{\gamma}{2}}}^2}
dL_N \, d\mu^\perp_{s;N},
$$
where $\11$ denotes the indicator function of a measurable set.
Using the Fubini theorem, we obtain that \eqref{expression} can be written as
$$
\int_{E_N}
e^{-\|\pi_N u\|_{H^{s+\frac{\gamma}{2}}}^2}
\Big(
\int_{E_N^\perp}\11(\Phi_N(t)(A))(\pi_{N}(u)+\pi_{>N}(u))\chi_r(\pi_N(u)+\pi_{>N}(u))d\mu^\perp_{s;N}\Big)dL_N,
$$
where $\pi_{>N}={\rm Id}-\pi_N$. 
Thanks to Lemma~\ref{rotation}, we can write the last expression equals
$$
\int_{E_N}
e^{-\|\pi_N u\|_{H^{s+\frac{\gamma}{2}}}^2}
\Big(
\int_{E_N^\perp}\11(\Phi_N(t)(A))(\pi_{N}(u)+S(t)\pi_{>N}(u))\chi_r(\pi_N(u)+S(t)\pi_{>N}(u))d\mu^\perp_{s;N}\Big)dL_N.
$$
Using once again the Fubini theorem, we get that \eqref{expression} equals
$$
\int_{E_N^\perp}
\Big(
\int_{E_N}
e^{-\|\pi_N u\|_{H^{s+\frac{\gamma}{2}}}^2}
\11(\Phi_N(t)(A))(\pi_{N}(u)+S(t)\pi_{>N}(u))\chi_r(\pi_N(u)+S(t)\pi_{>N}(u))dL_N\Big)d\mu^\perp_{s;N}.
$$
Now, using Lemma~\ref{liouvile}, we obtain that the last expression is equal to
\begin{multline*}
\int_{E_N^\perp}
\Big(
\int_{E_N}
e^{-\|\tilde{\Phi}_N(t)(\pi_N u)\|_{H^{s+\frac{\gamma}{2}}}^2}
\11(\Phi_N(t)(A))(\tilde{\Phi}_N(t)(\pi_{N}u)+S(t)\pi_{>N}(u))
\\
\chi_r(\tilde{\Phi}_N(t)(\pi_Nu)+S(t)\pi_{>N}(u))dL_N\Big)d\mu^\perp_{s;N}.
\end{multline*}
Coming back to \eqref{split}, we observe that $\tilde{\Phi}_N(t)(\pi_N u)=\pi_{N}\Phi_{N}(t)(u)$ and therefore
\eqref{expression} equals
\begin{equation*}
\int_{E_N^\perp}
\Big(
\int_{E_N}
e^{-\|\pi_{N}\Phi_N(t)(u)\|_{H^{s+\frac{\gamma}{2}}}^2}
\11(\Phi_N(t)(A))(\Phi_N(t)(u))
\chi_r(\Phi_N(t)(u))dL_N\Big)d\mu^\perp_{s;N}.
\end{equation*}
Since $\Phi_N(t)$ is a bijection, we have that $\11(\Phi_N(t)(A))(\Phi_N(t)(u))=\11(A)(u)$.
Moreover, using Lemma~\ref{boris}, we obtain that $\chi_r(\Phi_N(t)(u))=\chi_{r}(u)$.
Therefore, we finally obtain that 
\eqref{expression} equals
\begin{equation*}
\int_{E_N^\perp}
\Big(
\int_{E_N}
e^{-\|\pi_{N}\Phi_N(t)(u)\|_{H^{s+\frac{\gamma}{2}}}^2}
\11(A)(u)
\chi_r(u)dL_N\Big)d\mu^\perp_{s;N}
\end{equation*}
which equals
$$
\int_{A} 
\chi_{r}(u) 
e^{-\|\pi_N(\Phi_N(t)(u))\|^2_{H^{s+\frac{\gamma}{2}}}} du_1...du_N\, d\mu^\perp_{s;N}\,.
$$
This completes the proof of Proposition~\ref{ch-var}.
\end{proof}
\section{An energy estimate}
The following energy estimate is of importance in the study of the transport of the measure $\mu_s$ by $\Phi(t)$.
\begin{proposition}\label{energy}
Let $\gamma\in(4/3,2]$ and $s\geq 1$. 
Then there exist $\kappa<2$, $\varepsilon>0$ and a constant $C$ such that for every $N$ and every solution $u$  of \eqref{BBM-gamma-N},
$$
\frac{d}{dt}\|\pi_{N}u(t)\|_{H^{s+\gamma/2}}^2\leq 
C\,\big(1+\|\pi_{N}u(t)\|_{H^{\gamma/2}}^{3-\kappa}\big)
\big(1+\||D_x|^{s+\frac{\gamma}{2}-\frac{1}{2}-\varepsilon}
\pi_Nu(t)\|^\kappa_{L^{\infty}}\big)
$$
\end{proposition}
\begin{proof}
We first observe that $\pi_N u$ is a solution of 
\begin{equation*}
\partial_t \pi_N u+\partial_t|D_x|^\gamma \pi_N u+\partial_x \pi_N u+\partial_x\pi_{N}((\pi_{N}u)^2)=0\,.
\end{equation*}
Therefore, we can compute
\begin{multline*}
\frac{d}{dt}\|\pi_N u(t)\|_{H^{s+\gamma/2}}^2
=
\\
\frac{1}{2\pi}\int (|D_x|^{s+\frac{\gamma}{2}}\pi_N u)\,
 |D_x|^{s+\frac{\gamma}{2}}
 \big(
 -(1+|D_x|^{\gamma})^{-1}\partial_x \pi_N u
 -  (1+|D_x|^{\gamma})^{-1} \partial_x\pi_{N}(\pi_{N}u)^2
 \big).
\end{multline*}
Since 
$$
\int (|D_x|^{s+\frac{\gamma}{2}}\pi_N u)\,( |D_x|^{s+\frac{\gamma}{2}}(1+|D_x|^{\gamma})^{-1}\partial_x \pi_N u)=0
$$
and using that $\pi_N$ is a projector, we obtain that
\begin{eqnarray*}
\frac{d}{dt}\|\pi_N u(t)\|_{H^{s+\gamma/2}}^2
&= &
-\frac{1}{2\pi}\int (|D_x|^{s+\frac{\gamma}{2}}\pi_N u)\,
 |D_x|^{s+\frac{\gamma}{2}}
 \big( (1+|D_x|^{\gamma})^{-1} \partial_x(\pi_{N}u)^2
 \big)
 \\& = &
 -\frac{1}{2\pi}\int ((1+|D_x|^{\gamma})^{-1} |D_x|^{s+\gamma}\pi_N u)\,
 |D_x|^{s}
 \big(\partial_x(\pi_{N}u)^2 \big)\,.
\end{eqnarray*}
By writing 
$$
(1+|D_x|^{\gamma})^{-1} |D_x|^{\gamma}={\rm Id}-(1+|D_x|^{\gamma})^{-1}
$$
we arrive at
$$
\frac{d}{dt}\|\pi_N u(t)\|_{H^{s+\gamma/2}}^2=I_1+I_2,
$$
where 
$$
I_1=-\frac{1}{2\pi}\int (|D_x|^{s}\pi_N u)\,|D_x|^{s}
 \big(\partial_x(\pi_{N}u)^2 \big)
=
-\frac{1}{2\pi}\int (\partial_x^{s}\pi_N u)\,\partial_x^{s}
 \big(\partial_x(\pi_{N}u)^2 \big)
$$
and
$$
I_2=\frac{1}{2\pi}\int ((1+|D_x|^{\gamma})^{-1} |D_x|^{s}\pi_N u)\,
 |D_x|^{s}
 \big(\partial_x(\pi_{N}u)^2 \big)\,.
$$
Let us first estimate the more regular contribution of $I_2$.
Using that $\gamma> 1$ and $s>1/2$, we obtain that
$$
I_2\lesssim \|\pi_N u\|_{H^s}\|(\pi_N u)^2\|_{H^s}\lesssim \|\pi_N u\|_{H^s}^2\|\pi_N u\|_{L^\infty}
\lesssim \|\pi_N u\|_{H^s}^2\|\pi_N u\|_{H^{\gamma/2}}\,.
$$
Thanks to a suitable use of the H\"older inequality, we obtain that for some $\theta>0$
$$
\|\pi_N u\|_{H^s}\leq \|\pi_N u\|^\theta_{H^{\gamma/2}}
\|\pi_N u\|^{1-\theta}_{H^{s+\frac{\gamma}{2}-\frac{1}{2}-\varepsilon}},
$$
provided $\varepsilon$ is small enough. Now since our spatial domain is compact, we have that
$$
\|\pi_N u\|_{H^{s+\frac{\gamma}{2}-\frac{1}{2}-\varepsilon}}
\lesssim
\||D_x|^{s+\frac{\gamma}{2}-\frac{1}{2}-\varepsilon}\pi_N u\|_{L^\infty}\,.
$$
Therefore we arrive at 
$$
I_2\lesssim 
\|\pi_{N}u\|_{H^{\gamma/2}}^{1+2\theta}
\||D_x|^{s+\frac{\gamma}{2}-\frac{1}{2}-\varepsilon}
\pi_Nu\|^{2-2\theta}_{L^{\infty}}
$$
which is an acceptable bound. 

Let us next turn to the more delicate analysis of $I_1$. We can write
\begin{equation*}
I_1=-\frac{1}{\pi}\int \partial_x^sv \,\partial_x^s(\partial_x v v), 
\end{equation*}
where for shortness, we set $v=\pi_Nu$.
When applying the Leibniz rule, the most delicate term which appears is the one when all $s$ derivatives hit on $\partial_x v$.
Namely, we have to deal with the term 
\begin{equation}\label{EE}
\int (\partial_x^sv) \,(\partial_x^{s+1}v)\, v\,.
\end{equation}
In the spirit of the local well-posedness theory of quasilinear hyperbolic PDE's, the main observation is that one may rewrite \eqref{EE} as
$$
-\frac{1}{2}\int \partial_xv (\partial_x^s v)^2\,.
$$
Therefore, thanks to the last key argument and the Leibniz rule, we obtain that in order to estimate $I_1$, it suffices to estimate the expressions 
\begin{equation}\label{M1}
\int (\partial_x^s v)(\partial_x^{\sigma_1}v)(\partial_x^{\sigma_2}v),
\end{equation}
where 
$$
\sigma_1+\sigma_2=s+1,\quad \sigma_1\leq s,\,\, \sigma_2\leq s,
$$
(the important point being that $\sigma_1$ and $\sigma_2$ are not allowed to be $s+1$).
For that purpose, we will use the following lemma.
\begin{lemme}\label{LP}
Let $\sigma\in [\frac{\gamma}{2},s+\frac{\gamma}{2}-\frac{1}{2}-\varepsilon]$. 
Suppose that $\theta\in[0,1]$ is such that
\begin{equation}\label{asss}
\sigma<\theta\frac{\gamma}{2}+(1-\theta)(s+\frac{\gamma}{2}-\frac{1}{2}-\varepsilon)
\end{equation}
Then for $u$ such that $\hat{u}(0)=0$, we have the bound
$$
\|\partial_x^\sigma u\|_{L^p}\lesssim
\|u\|^\theta_{H^{\gamma/2}}
\|| D_x|^{s+\frac{\gamma}{2}-\frac{1}{2}-\varepsilon}u\|^{1-\theta}_{L^\infty}\,,
$$
provided $\varepsilon>0$ is sufficiently small, where 
$
\frac{1}{p}=\frac{\theta}{2}+\frac{1-\theta}{\infty}$,  i.e. $p=\frac{2}{\theta}$.
\end{lemme}
\begin{proof}
Consider a Littlewood-Paley decomposition of the unity
\begin{equation}\label{lp1}
{\rm Id}=\sum_{\lambda}\Delta_{\lambda},
\end{equation}
where the summation is taken over the dyadic values of $\lambda$, i.e. $\lambda=2^j$, $j=0,1,2,\dots$ and $\Delta_{\lambda}$ are 
Littlewood-Paley projectors. More precisely they are defined as Fourier multipliers as $\Delta_0=\psi_0(|D_x|)$ and for $\lambda\geq 1$,
$\Delta_\lambda=\psi(|D_x|/\lambda)$, where $\psi_0\in C_0^\infty(-\frac{1}{2},\frac{1}{2})$ and $\psi\in C_0^\infty(\R\backslash\{0\})$ are suitable functions such that
\eqref{lp1} holds. 
In the sequel, we shall use that for every $\varphi \in C_0^\infty(\R\backslash\{0\})$ one has the bound
\begin{equation}\label{lp2}
\|\varphi(|D_x|/\lambda)(f)\|_{L^p}\leq C\|f\|_{L^p},\quad p\in [1,\infty],\,\, \lambda\geq 1,
\end{equation}
where the constant $C$ is {\it independent} of $\lambda$.  The bound \eqref{lp2} is a consequence of the Schur lemma.
Indeed, one needs to invoke the following estimate for the kernel of $\varphi(|D_x|/\lambda)$,
$$
\Big|
\sum_{n}\varphi\Big(\frac{|n|}{\lambda}\Big)e^{in(x-y)}
\Big|\leq \frac{C\lambda}{(1+\lambda|x-y|)^2}
$$
which follows after two summations by parts. We notice that the extension of \eqref{lp2} when the circle is replaced by a compact Riemannian manifold is known 
to hold (see e.g. \cite{BGT}).

For $u$ such that $\hat{u}(0)=0$, we have  $\Delta_0(u)=0$ and therefore 
\begin{equation}\label{lp3}
\|\partial_x^\sigma u\|_{L^p}\leq \sum_{\lambda\geq 1}\|\partial_x^\sigma \Delta_{\lambda}u\|_{L^p}\,.
\end{equation}
Similarly to \eqref{lp2}, using the Schur lemma, we can write
\begin{equation}\label{lp4}
 \|\partial_x^\sigma \Delta_{\lambda}u\|_{L^p}\lesssim \lambda^\sigma \|\Delta_{\lambda}u\|_{L^p}.
\end{equation}
Using \eqref{lp3}, \eqref{lp4} and the H\"older inequality, we arrive at the bound
\begin{equation}\label{lp5}
\|\partial_x^\sigma u\|_{L^p}\lesssim \sum_{\lambda\geq 1}
\lambda^\sigma
\|\Delta_{\lambda}u\|_{L^2}^\theta
\|\Delta_{\lambda} u\|_{L^\infty}^{1-\theta}\,,
\end{equation}
where $p=2/\theta$. 
Now we can write
$$
\lambda^{s+\frac{\gamma}{2}-\frac{1}{2}-\varepsilon}\Delta_{\lambda}=
\tilde{\psi}(|D_x|/\lambda)|D_x|^{s+\frac{\gamma}{2}-\frac{1}{2}-\varepsilon},
$$
where 
$\tilde{\psi}\in C_0^\infty(\R\backslash\{0\})$ is chosen such that
for $x\geq 0$, $\tilde{\psi}(x)=\psi(x)/x^{s+\frac{\gamma}{2}-\frac{1}{2}-\varepsilon}$.
Therefore, using \eqref{lp2}, we get
$$ 
\|\Delta_{\lambda} u\|_{L^\infty}
\lesssim 
\lambda^{-(s+\frac{\gamma}{2}-\frac{1}{2}-\varepsilon)}
\|| D_x|^{s+\frac{\gamma}{2}-\frac{1}{2}-\varepsilon}u\|_{L^\infty}\,.
$$
Similarly, one obtains 
$$
\|\Delta_{\lambda}u\|_{L^2}
\lesssim
\lambda^{-\frac{\gamma}{2}}
\| |D_x|^{\frac{\gamma}{2}}u\|_{L^2}\,.
$$
Therefore, coming back to \eqref{lp5}, we get the bound
$$
\|\partial_x^\sigma u\|_{L^p}
\lesssim
\|| D_x|^{\frac{\gamma}{2}}u\|^\theta_{L^2}
\|| D_x|^{s+\frac{\gamma}{2}-\frac{1}{2}-\varepsilon}u\|^{1-\theta}_{L^\infty}
\Big(
\sum_{\lambda\geq 1}
\lambda^{\sigma-\theta\frac{\gamma}{2}-(1-\theta)(s+\frac{\gamma}{2}-\frac{1}{2}-\varepsilon)}
\Big)\,.
$$
Thanks to \eqref{asss} the sum appearing in the right hand-side of the last inequality is convergent. Thus we arrive at the needed bound.
This completes the proof of Lemma~\ref{LP}.
\end{proof}
Let us estimate the expressions \eqref{M1}. Suppose first that $s\geq 2$.
For $\varepsilon_1$ and $\varepsilon$ sufficiently small to be fixed, we define the numbers $\theta_1,\theta_2,\theta_3$ in the interval $(0,1)$ as
\begin{eqnarray}\label{smetki}
s+\varepsilon_1& = &\theta_1\frac{\gamma}{2}+(1-\theta_1)(s+\frac{\gamma}{2}-\frac{1}{2}-\varepsilon)
\\
\label{smetki2}
\sigma_1+\varepsilon_1& = &\theta_2\frac{\gamma}{2}+(1-\theta_2)(s+\frac{\gamma}{2}-\frac{1}{2}-\varepsilon)
\\\label{smetki3}
\sigma_2+\varepsilon_1& = &\theta_3\frac{\gamma}{2}+(1-\theta_3)(s+\frac{\gamma}{2}-\frac{1}{2}-\varepsilon)
\end{eqnarray}
(observe that $s,\sigma_1,\sigma_2\in [1,s]$ and $1/2\leq \gamma/2\leq 1$).
We next define $p_1$, $p_2$, $p_3$ as $p_j=2/\theta_j$, $j=1,2,3$. We now check that under our assumptions of $\gamma$ and $s$ for 
$\varepsilon_1$ and $\varepsilon$ sufficiently small, we have
\begin{equation}\label{holder}
\frac{1}{p_1}+\frac{1}{p_2}+\frac{1}{p_3}\leq 1.
\end{equation}
Clearly \eqref{holder} is equivalent to
$\theta_1+\theta_2+\theta_3\leq2$.
But coming back to \eqref{smetki}, \eqref{smetki2}, \eqref{smetki3} we obtain that
$$
\theta_1+\theta_2+\theta_3=\frac{s+\frac{3\gamma}{2}-\frac{5}{2}-3\varepsilon_1-3\varepsilon}{s-\frac{1}{2}-\varepsilon}\,.
$$
Therefore for  $\varepsilon_1$ and $\varepsilon$ sufficiently small the condition \eqref{holder} follows from $s+\frac{3}{2}>\frac{3\gamma}{2}$, which is satisfied thanks 
to the assumption $s\geq 2$.

Thanks to \eqref{holder}, we can apply the H\"older inequality  and Lemma~\ref{LP} to write
\begin{eqnarray*}
\Big|\int (\partial_x^s v)(\partial_x^{\sigma_1}v)(\partial_x^{\sigma_2}v)\Big|
&
\leq &
\|\partial_x^s v\|_{L^{p_1}}
\|\partial_x^{\sigma_1}v\|_{L^{p_2}}
\|\partial_x^{\sigma_2}v\|_{L^{p_3}}
\\
& \lesssim & 
\|v\|_{H^{\gamma/2}}^{\theta_1+\theta_2+\theta_3}
\|
|D_x|^{s+\frac{\gamma}{2}-\frac{1}{2}-\varepsilon}
v\|^{3-\theta_1-\theta_2-\theta_3}_{L^{\infty}}\,.
\end{eqnarray*}
Therefore it remains to verify that $\theta_1+\theta_2+\theta_3>1$.
But  for  $\varepsilon_1$ and $\varepsilon$ sufficiently small this condition follows from our assumption $\gamma>4/3$.

Let us finally consider the case $s=1$ which is not covered by the above analysis. In this case, we only need to estimate
$
\int (\partial_x v)^3.
$
Therefore in the case $s=1$ one gets an acceptable bound for $I_1$ thanks to the following lemma.
\begin{lemme}\label{LP-pak}
For every $\gamma>4/3$ there is $\theta>1/3$  such that for $u$ satisfying $\hat{u}(0)=0$, we have the bound
$$
\|\partial_x u\|_{L^3}\lesssim
\|u\|^\theta_{H^{\gamma/2}}
\|| D_x|^{\frac{1}{2}+\frac{\gamma}{2}-\varepsilon}u\|^{1-\theta}_{L^\infty}\,,
$$
provided $\varepsilon>0$ is sufficiently small.
\end{lemme}
\begin{proof}
As in the proof of Lemma~\ref{LP}, we perform a Littlewood-Paley decomposition. 
We can write
\begin{equation*}
\|\partial_x u\|_{L^3}\lesssim \sum_{\lambda\geq 1}
\lambda
\|\Delta_{\lambda}u\|_{L^2}^{\frac{2}{3}}
\|\Delta_{\lambda} u\|_{L^\infty}^{\frac{1}{3}}\,.
\end{equation*}
We now choose $\sigma$ such that
\begin{equation}\label{B1}
\frac{2}{3}\sigma+\frac{1}{3}\big(\frac{1}{2}+\frac{\gamma}{2}-\varepsilon\big)>1\,.
\end{equation}
Thanks to \eqref{B1} as in the proof of Lemma~\ref{LP}, we arrive at the bound
\begin{equation}\label{bea}
\|\partial_x u\|_{L^3}\lesssim
\|u\|^{\frac{2}{3}}_{H^{\sigma}}
\|| D_x|^{\frac{1}{2}+\frac{\gamma}{2}-\varepsilon}u\|^{\frac{1}{3}}_{L^\infty}\,,
\end{equation}
We choose more precisely $\sigma$ such that
$$
\frac{2}{3}\sigma+\frac{1}{3}\big(\frac{1}{2}+\frac{\gamma}{2}-\varepsilon\big)=1+\varepsilon
$$
which leads to
$$
\sigma=\frac{5-\gamma+8\varepsilon}{4}\,.
$$
We first observe that as far as $\gamma>1$ for $\varepsilon$ small enough $\sigma<\frac{1}{2}+\frac{\gamma}{2}-\varepsilon$.
If $\sigma \leq \gamma/2$ then the bound \eqref{bea} is already sufficient to complete the proof of Lemma~\ref{LP-pak} (with $\theta=\frac{2}{3}$).
We can therefore suppose that $\sigma\in [\frac{\gamma}{2}, \frac{1}{2}+\frac{\gamma}{2}-\varepsilon]$.
Now, thanks to a suitable use of the H\"older inequality, we can write
$$
\|u\|_{H^{\sigma}}\leq
\|u\|^{\alpha}_{H^{\gamma/2}}
\|u\|^{1-\alpha}_{H^{\frac{1}{2}+\frac{\gamma}{2}-\varepsilon}}
\lesssim
\|u\|^{\alpha}_{H^{\gamma/2}}
\||D_x|^{\frac{1}{2}+\frac{\gamma}{2}-\varepsilon}u\|^{1-\alpha}_{L^\infty}\,,
$$
where
$$
\sigma=\alpha\frac{\gamma}{2}+(1-\alpha)\big(\frac{1}{2}+\frac{\gamma}{2}-\varepsilon\big)\,.
$$
Now the claim of the lemma follows if we can assure that 
$
\frac{2}{3}\alpha>\frac{1}{3}.
$
A direct computation shows that the last inequality is equivalent to
$\gamma>\frac{4}{3}+\frac{10\varepsilon}{3}$ which can be assured for $\varepsilon$ small enough, thanks to our assumption $\gamma>\frac{4}{3}$.
This completes the proof of Lemma~\ref{LP-pak}
\end{proof}
Summarizing the previous discussion provides the needed bound for $I_1$.
This completes the proof of Proposition~\ref{energy}.
\end{proof}
\begin{remarque}
Let us observe that the for $\gamma>2$ one may obtain the analogue of Proposition~\ref{energy} by using an argument which does not require the integration by parts 
trick on
the quantity \eqref{EE}. Indeed for $\gamma>2$, the expression 
$$
\int ((1+|D_x|^{\gamma})^{-1} |D_x|^{s+\gamma}\pi_N u)\,
 |D_x|^{s}
 \big(\partial_x(\pi_{N}u)^2 \big)
 $$
 has enough smoothing so that we can employ a semi-linear technique to achieve the desired bound. 
 Interestingly, $\gamma=2$ is also the border  line of the applicability of the result of \cite{Ra}.
\end{remarque}
\section{A large deviation bound}
We can now invoke the following large deviation estimate for the quantity appearing in the energy estimate.
\begin{lemme}\label{LD}
Let $\varepsilon>0$. There exists $C$ such that for every $r>0$, every $p\geq 2$, every $N\geq 1$,
$$
\Big\|
\|| D_x|^{s+\frac{\gamma}{2}-\frac{1}{2}-\varepsilon}\pi_N u\|_{L^{\infty}}
\Big\|_{L^{p}(\mu_{s,r}(u))}
\leq 
Cp^{\frac{1}{2}}\,.
$$
\end{lemme}
\begin{proof}
We first observe that
$$
\Big\|
\|| D_x|^{s+\frac{\gamma}{2}-\frac{1}{2}-\varepsilon}\pi_N u\|_{L^{\infty}}
\Big\|_{L^{p}(\mu_{s,r}(u))}
\leq 
\Big\|
\|| D_x|^{s+\frac{\gamma}{2}-\frac{1}{2}-\varepsilon}\pi_N u\|_{L^{\infty}}
\Big\|_{L^{p}(\mu_{s}(u))}\,.
$$
Therefore, we need to prove that
$$
\Big\|
\|
|D_x|^{s+\frac{\gamma}{2}-\frac{1}{2}-\varepsilon}\pi_N u\|_{L^{\infty}}
\Big\|_{L^{p}(\mu_{s}(u))}
\leq 
Cp^{\frac{1}{2}}\,.
$$
Coming back to the definition of $\mu_s$, the last inequality can be rewritten as
$$
\Big\|
\|
|D_x|^{s+\frac{\gamma}{2}-\frac{1}{2}-\varepsilon}
\sum_{n\neq 0, |n|\leq N}\frac{g_n(\omega)}{|n|^{s+\gamma/2}}e^{inx}\|_{L^\infty}
\Big\|_{L^{p}_\omega}
\leq 
Cp^{\frac{1}{2}}\,.
$$
which would follow from
$$
\Big\|\|\sum_{n\neq 0, |n|\leq N}\frac{g_n(\omega)}{|n|^{\frac{1}{2}+\varepsilon}}e^{inx}\|_{L^{\infty}}
\Big\|_{L^{p}_\omega}
\leq 
Cp^{\frac{1}{2}}\,.
$$
Now, using the Sobolev embedding, we obtain that the last inequality follows from
\begin{equation}\label{last}
\Big\|\|\sum_{n\neq 0, |n|\leq N}\frac{g_n(\omega)}{|n|^{\frac{1}{2}+\frac{\varepsilon}{2}}}e^{inx}\|_{L^{q}}
\Big\|_{L^{p}_\omega}
\leq 
Cp^{\frac{1}{2}}\,,
\end{equation}
provided $q>\frac{2}{\varepsilon}$.
The inequality \eqref{last} is classical (see e.g. \cite[Lemma~3.1]{BT1/2}). 
This completes the proof of Lemma~\ref{LD}.
\end{proof}
\section{The measure evolution property}
\begin{lemme}\label{dve}
There is  $0\leq \beta<1$ such that for every $r>0$ there is a constant $C>0$ such that for every $p\geq 2$ and every Borel set $A$ of $H^s$, every $N\geq 1$,
$$
\frac{d}{dt}\mu_{s,r}(\Phi_N(t)(A))\leq  Cp^\beta(\mu_{s,r}(\Phi_N(t)(A)))^{1-\frac{1}{p}} \,.
$$
\end{lemme}
\begin{remarque}
We remark that in a similar situation in \cite{TV}, the measure $(\mu_{s,r}(\Phi_N(t)(A)))^{1-\frac{1}{p}}$ is simply bounded by one. Thanks to the use of much more subtle energies related to the remarkable but very particular structure of the Benjamin-Ono equation,  in \cite{TV} the contribution corresponding to $Cp^\beta$ is a (delicate)
quantity tending to zero as $N\rightarrow\infty$ (which lead to the invariance of the corresponding measure while here we only get quasi-invariance).
\end{remarque}
\begin{proof}[Proof of Lemma~\ref{dve}]
Using the flow properties of $\Phi_N(t)$, we obtain that
\begin{eqnarray*}
\frac{d}{dt}\mu_{s,r}(\Phi_N(t)(A))\Big|_{t=\bar{t}}
&=&
\frac d{dt} \int_{\Phi_N(t)(A)} \chi_r(u) d\mu_{s,r}(u)
\Big |_{t=\bar t}
\\
& = & 
\frac d{dt}\int_{\Phi_N(t)(\Phi_N(\bar{t})(A))} 
\chi_r(u)
 d\mu_{s,r}(u)\Big |_{t=0}\equiv I.
 \end{eqnarray*}
Using Proposition~\ref{ch-var}, we can write $I$ as 
$$
I=\gamma_N 
\frac d{dt} \int_{\Phi_N(\bar{t})(A)}
\chi_r(u) 
e^{-\|\pi_N(\Phi_N(t)(u))\|^2_{H^{s+\gamma/2}}}\,\, du_1...du_N\, d\mu^\perp_{s;N}  
\Big |_{t=0}\,.
$$
Therefore
\begin{multline*}
I=-\gamma_N
\int_{\Phi_N(\bar{t})(A)} \chi_{r}(u)
\Big(
\frac{d}{dt}
\|\pi_N(\Phi_N(t)(u))\|^2_{H^{s+\frac{\gamma}{2}}}\Big|_{t=0}
\Big)
\\ 
e^{-\|\pi_N(u)\|^2_{H^{s+\gamma/2}}}\,\, du_1...du_N\, d\mu^\perp_{s;N}  
\end{multline*}
and consequently
$$
I=-\int_{\Phi_N(\bar{t})(A)}\Big(
\frac{d}{dt}
\|\pi_N(\Phi_N(t)(u))\|^2_{H^{s+\frac{\gamma}{2}}}\Big|_{t=0}
\Big)d\mu_{s,r}(u)\,.
$$
Therefore, using the H\"older inequality, we can write
$$
I
\leq 
\Big\|\frac{d}{dt}
\|\pi_N(\Phi_N(t)(u))\|^2_{H^{s+\frac{\gamma}{2}}}\Big|_{t=0}\Big\|_{L^p(\mu_{s,r}(u))}
(\mu_{s,r}(\Phi_N(\bar{t})(A)))^{1-\frac{1}{p}} \,.
$$
Therefore, it remains to show that 
$$
\Big\|\frac{d}{dt}
\|\pi_N(\Phi_N(t)(u))\|^2_{H^{s+\frac{\gamma}{2}}}\Big|_{t=0}\Big\|_{L^p(\mu_{s,r}(u))}
\leq 
Cp^\beta,
$$
for some $\beta<1$.
At this point, we shall invoke the energy estimate of Proposition~\ref{energy}.
Namely, using Proposition~\ref{energy}, we can write
\begin{multline}\label{krai}
\Big\|\frac{d}{dt}
\|\pi_N(\Phi_N(t)(u))\|^2_{s+\gamma/2}\Big|_{t=0}\Big\|_{L^p(\mu_{s,r}(u))}
\\
\leq 
C\Big\|
\big(1+\|\pi_N\Phi_N(t)(u)\|_{H^{\gamma/2}}^{3-\kappa}\big)
\big(1+\| |D_x|^{s+\frac{\gamma}{2}-\frac{1}{2}-\varepsilon}
\pi_N\Phi_N(t)(u)\|^\kappa_{L^{\infty}}\big)
\Big|_{t=0}\Big\|_{L^p(\mu_{s,r}(u))}
\\
\leq C(1+r^{3-\kappa})
\big(1+
\Big\|
\|
| D_x|^{s+\frac{\gamma}{2}-\frac{1}{2}-\varepsilon}
\pi_N u\|_{L^{\infty}}
\Big\|^\kappa_{L^{\kappa p}(\mu_{s,r}(u))}\Big),
\end{multline}
for some $\kappa<2$.

Using Lemma~\ref{LD} and \eqref{krai}, we obtain that
$$
\Big\|\frac{d}{dt}
\|\pi_N(\Phi_N(t)(u))\|^2_{H^{s+\frac{\gamma}{2}}}\Big|_{t=0}\Big\|_{L^p(\mu_{s,r}(u))}
\leq 
Cp^{\frac{\kappa}{2}}\,.
$$
This completes the proof of Lemma~\ref{dve}.
\end{proof}
We are now in position to apply a variant the Yudovich argument (\cite{Yu}).
\begin{lemme}\label{edno}
Fix $t\in\R$, $r\geq 0$ and $\delta>0$. 
There exists $C>0$ such that for every Borel set $A$ of $H^{s}$, every $N\geq 1$, 
$\mu_{s,r}(\Phi_N(t)(A))\leq C(\mu_{s,r}(A))^{1-\delta}$.
\end{lemme}
\begin{proof}
The conclusion of Lemma~\ref{dve}, can be written as
$$
\frac{d}{dt}\big(\mu_{s,r}(\Phi_N(t)(A))\big)^{\frac{1}{p}}
\leq Cp^{-\alpha},
$$
where $\alpha=1-\beta>0$. After an integration, we obtain that
\begin{eqnarray*}
\mu_{s,r}(\Phi_N(t)(A)) & \leq & \big((\mu_{s,r}(A))^{\frac{1}{p}}+Ctp^{-\alpha}\big)^{p}
\\
& = &
\mu_{s,r}(A)e^{p\log\big(1+Ctp^{-\alpha}(\mu_{s,r}(A))^{-\frac{1}{p}}\big)}\,.
\end{eqnarray*}
Using that for $x\geq 0$, one has $\log(1+x)\leq x$, we arrive at the bound
$$
\mu_{s,r}(\Phi_N(t)(A))\leq 
\mu_{s,r}(A)e^{Ctp^{1-\alpha}(\mu_{s,r}(A))^{-\frac{1}{p}}}\,.
$$
We now choose $p$ as
$$
p\equiv2+\log\Big(\frac{1}{\mu_{s,r}(A)}\Big). 
$$
Therefore we obtain that 
\begin{equation}\label{uns}
\mu_{s,r}(\Phi_N(t)(A))\leq \mu_{s,r}(A)e^{Cet\big(2+\log\big(\frac{1}{\mu_{s,r}(A)}\big)\big)^{1-\alpha}}\,.
\end{equation}
We therefore conclude that for every $\delta>0$ there is a constant $\tilde{C}=\tilde{C}(\delta,\alpha,C,t)$ (i.e. depending also on $\alpha,C$ and $t$) such that
$$
\mu_{s,r}(\Phi_N(t)(A))\leq \tilde{C}(\delta,\alpha,C,t)(\mu_{s,r}(A))^{1-\delta}\,.
$$
This completes the proof of Lemma~\ref{edno}.
\end{proof}
\begin{remarque}\label{growth}
Using the argument of \cite{B} the bound \eqref{uns} may be used to obtain that the $H^{s+\frac{\gamma}{2}-\frac{1}{2}-\varepsilon}$ norms of the solutions with data on the support of $\mu_s$ do not grow faster than a quantity of type $t^{\gamma(s)}$ ($t\gg 1$) with $\gamma(s)\rightarrow \infty$ as $s\rightarrow\infty$ ($\gamma(s)$ may be taken close to $s-\frac{1}{2}$). However, such a bound may be achieved by purely deterministic methods. On the other hand, if the bound \eqref{uns} is replaced by the stronger bound 
$$ 
\mu_{s,r}(\Phi_N(t)(A))\leq \mu_{s,r}(A)e^{Ct}
$$
then the argument of \cite{B} would give a bound of type $t^{\frac{1}{2}}$. Such a bound would be of greater interest because the power is independent of $s$.
\end{remarque}
\section{End of the proof of the main result}
\begin{lemme}\label{lyon}
Fix $t\in\R$, $r>0$, $R>0$ and $\delta>0$. There exists $C>0$ such that for every Borel set $A\subset B_{R,s}$ of $H^{s}$,  
$\mu_{s,r}(\Phi(t)(A))\leq C(\mu_{s,r}(A))^{1-\delta}$.
\end{lemme}
\begin{proof}
We first show the statement of Lemma~\ref{lyon} if $A\subset B_{R,s}$ is a compact set.  In this case, using  Proposition~\ref{approximation} and Proposition~\ref{edno}, we obtain that for every $\varepsilon>0$
there is $N_0$ such that for every $N\geq N_0$,
$$
\mu_{s,r}(\Phi(t)(A))\leq \mu_{s,r}(\Phi_N(t)(A+B_{\varepsilon,s}))\leq C(\delta, t,r)\big( \mu_{s,r}(A+B_{\varepsilon,s})\big)^{1-\delta}\,.
$$
Now, since $A$ is a compact, using the dominate convergence theorem, we obtain that in the limit $\varepsilon\rightarrow 0$,
$$
\mu_{s,r}(\Phi(t)(A))\leq  C(\delta, t,r)\big( \mu_{s,r}(A)\big)^{1-\delta}\,.
$$
Let now $A\subset B_{R,s}$ be an arbitrary Borel set. Using the regularity of $\mu_{s,r}$, we obtain that there is a sequence $(K_n)_{n=1}^{\infty}$ of compacts of $H^s$ such that $K_n\subset \Phi(t)(A)$ and 
\begin{equation}\label{f0}
\lim_{n\rightarrow\infty}\mu_{s,r}(K_n)=\mu_{s,r}(\Phi(t)(A))\,.
\end{equation}
We next observe that
\begin{equation}\label{f1}
K_{n}\subset \Phi(t)(\Phi(-t)(K_n))\,.
\end{equation}
Indeed, every $x\in K_n$ can be written as 
$$
x=\Phi(t)(\Phi(-t)(x))\in \Phi(t)(\Phi(-t)(K_n))\,.
$$ 
Thus we have \eqref{f1}. 
Using \eqref{f1}, we infer that
\begin{equation}\label{f2}
\mu_{s,r}(K_n)\leq \mu_{s,r}(\Phi(t)(F_n)),\quad F_n=\Phi(-t)(K_n)\,.
\end{equation}
Since $K_n$ is a compact and $\Phi(-t)$ a continuous map, we obtain that $F_n$ is a compact. We now claim that
\begin{equation}\label{f3}
F_n\subset A\,.
\end{equation}
Indeed, let $x\in F_n$. This means that there is $y\in K_n$ such that $x=\Phi(-t)(y)$. But $K_n\subset \Phi(t)(A)$ and therefore $y\in \Phi(t)(A)$.
As a consequence, there exists $z\in A$ such that $y=\Phi(t)(z)$. Since $x=\Phi(-t)(y)$ and  $y=\Phi(t)(z)$, we infer that $x=z$. Therefore $x\in A$ and the proof of 
\eqref{f3} is complete. Thanks to the analysis for compact sets performed in the beginning of the proof, we obtain that
$$
\mu_{s,r}(\Phi(t)(F_n))\leq  C(\delta, t,r)\big( \mu_{s,r}(F_n)\big)^{1-\delta}
\leq C(\delta, t,r)\big( \mu_{s,r}(A)\big)^{1-\delta}
\,.
$$
Coming back to \eqref{f2}, we obtain the bound
$$
\mu_{s,r}(K_n)\leq C(\delta, t,r)\big( \mu_{s,r}(A)\big)^{1-\delta}\,.
$$
Passing to the limit $n\rightarrow \infty$ by invoking \eqref{f0} gives
$$
\mu_{s,r}(\Phi(t)(A))\leq C(\delta, t,r)\big( \mu_{s,r}(A)\big)^{1-\delta}\,.
$$
This completes the proof of Lemma~\ref{lyon}.
\end{proof}
Let us now complete the proof of Theorem~\ref{main}.
Let $A$ be a Borel set of $H^s$ such that $\mu_s(A)=0$. Our goal is to show that $\mu_{s}(\Phi(t)(A))=0$.
Since $\mu_s(A)=0$, we also have that for every $R, r>0$, $\mu_{s,r}(A\cap B_{R,s})=0$.
Therefore, thanks to  Lemma~\ref{lyon} for every $R,r>0$, $\mu_{s,r}(\Phi(t)(A\cap B_{R,s}))=0$.
On the other hand, thanks to the dominated convergence theorem, for every Borel set $A$ of $H^s$,
$$
\mu_s(A)=\lim_{r\rightarrow\infty}\mu_{s,r}(A).
$$
This implies that for every $R>0$, $\mu_{s}(\Phi(t)(A\cap B_{R,s}))=0$.
Now, we invoke the straightforward property
$$
\Phi(t)(A)=\bigcup_{R=1}^\infty \Phi(t)(A\cap B_{R,s}))
$$
to obtain that $\mu_{s}(\Phi(t)(A))=0$. This completes the proof of Theorem~\ref{main}. 
\section{Proof of Proposition~\ref{CM}}
We have that 
$
\Sigma(t)(u)=S(t)(u)+f(t),
$
where $f(t)$ is given by
$$
f(t)\equiv -\int_{0}^{t}S(t-\tau)\big((1+|D_x|^{\gamma})^{-1}\partial_x(h)\big)d\tau.
$$
We have that $f(t)\in H^\sigma$, thanks to the hypothesis on $h$.
Since $\mu_s$ is invariant under $S(t)$, we have that for every $g\in L^1(d\mu_s(u))$,
$$
\int_{H^s}g(S(t)(u)+f(t))d\mu_s(u)=\int_{H^s}g(u+f(t))d\mu_s(u)
$$
and therefore, we need to show that the image measure of $\mu_s$ under the map $u\mapsto u+f(t)$ (seen as a bijection on $H^\sigma$) 
is singular with respect to $\mu_s$.
We now show that 
\begin{equation}\label{diverge}
f(t)\notin H^{s+\frac{\gamma}{2}},\quad t\neq 0.
\end{equation}
We have 
$$
S(t-\tau)\big((1+|D_x|^{\gamma})^{-1}\partial_x(h)\big)=
\sum_{n}
\frac{in}{1+|n|^{\gamma}}\hat{h}(n)e^{inx}e^{-i(t-\tau)\frac{n}{1+|n|^\gamma}}
$$
and using that for $t\neq 0$ there is $c>0$ such that for every $n$
$$
\Big|
\int_{0}^t
e^{-i(t-\tau)\frac{n}{1+|n|^\gamma}}\,d\tau
\Big|\geq c,
$$
we obtain that for $\gamma<3/2$,
$$
\|f(t)\|_{H^{s+\frac{\gamma}{2}}}\geq c \|h\|_{H^{s-\frac{\gamma}{2}+1}}\geq c \|h\|_{H^{s+\frac{\gamma}{2}-\frac{1}{2}}}=+\infty.
$$
This completes the proof of \eqref{diverge}.
Therefore we can apply the Cameron-Martin argument as we now explain.
Thanks to \eqref{diverge}, there is $k\in H^{s+\frac{\gamma}{2}}$ such that
\begin{equation}\label{k0}
\sum_{n\neq 0}|n|^{2(s+\frac{\gamma}{2})}\widehat{f(t)}(n)\overline{\hat{k}(n)}=+\infty\,.
\end{equation}
The existence of $k$ may be obtained either by invoking the Banach-Steinhaus theorem or by an explicit construction. 
Next, we observe that
\begin{equation}\label{k1}
\mu_s\big(u\,:\, \sum_{n\neq 0}|n|^{2(s+\frac{\gamma}{2})}\hat{u}(n)\overline{\hat{k}(n)}<\infty\big)=
p\big(\omega\,:\, \sum_{n\neq 0}|n|^{s+\frac{\gamma}{2}}g_{n}(\omega)\overline{\hat{k}(n)}<\infty\big)=1,
\end{equation}
where we used that the basic orthogonality between $(g_n)$ yields 
$$
\Big\|
\sum_{n\neq 0}|n|^{s+\frac{\gamma}{2}}g_{n}(\omega)\overline{\hat{k}(n)}
\Big\|_{L^2_{\omega}}
\lesssim
\|k\|_{H^{s+\frac{\gamma}{2}}}\,.
$$
Thanks to \eqref{k1} there is a set $A\subset H^s$ such that $\mu_s(A)=1$ and for every $v\in A$,
\begin{equation}\label{k3}
\sum_{n\neq 0} |n|^{2(s+\frac{\gamma}{2})}\hat{v}(n)\overline{\hat{k}(n)}<\infty\,.
\end{equation}
Let us denote by $\mu_s^t$ the image measure of $\mu_s$ under the map $u\mapsto u+f(t)$.
Then
$$
\mu_s^t(A)=\mu_s(B),\quad
B\equiv\{v-f(t), v\in A\}.
$$
Thanks to \eqref{k0} and \eqref{k3}, we obtain that for every $u\in B$,
$$
\sum_{n\neq 0} |n|^{2(s+\frac{\gamma}{2})}\hat{u}(n)\overline{\hat{k}(n)}=\infty\,.
$$
Therefore $B\subset A^c$ and consequently $\mu_s(B)=0$, i.e. $\mu_s^t(A)=0$. Since $\mu_s(A)=1$, we conclude that $\mu_s$ and $\mu_s^t$ are mutually singular. 
This completes the proof of Proposition~\ref{CM}\,.
\section{Final remarks}
The arguments we presented here can be seen as a combination of the use of higher order pseudo-energies and the idea of \cite{TV} reducing the analysis of the transported measure to a property of the random series describing the set of the initial data. In this work we presented this approach in the simplest significant setting 
we found, namely the generalized BBM models. 
It would be interested to decide how much the results obtained here can be extended to other Hamiltonian PDE. 
For instance, we believe that a slight modification of the proof of Theorem~\ref{main} gives the quasi-invariance of the gaussian measures $\mu_s$ under the flow of the 1d Klein-Gordon equation
\begin{equation}\label{KG}
\partial_t^2u-\partial_x^2u+u+u^3=0\,.
\end{equation}
Such a result would however be at the border line of the Ramer result (there is $1$ smoothing when rewriting \eqref{KG} as a first order order equation)  and moreover it does not go beyond the Cameron-Martin threshold. Consequently, we find it less interesting than Theorem~\ref{main}.
The extension to the $2d$ in the context of \eqref{KG} is an interesting issue which is out of our present understanding of this set of problems.
Another issue which may be interesting is whether one may incorporate a dispersive effect in the measure quasi-invariance problems, i.e. whether one may exploit more subtle smoothing properties related to dispersion (see e.g. \cite{Bo,BIT,ET}). Finally, it would be very interesting to find situations where we can prove that the transported measure is singular with respect to the initial gaussian measure and describe the measure evolution. 
\section{Acknowledgements.} 
I am indebted to Nicolas~Burq and Nicola~Visciglia since this work benefited from our collaborations on related topics. 
I am grateful to Ana Bela~Cruzeiro, Huang~Guan and Tadahiro~Oh for  discussions on the subject discussed in this paper.  

\end{document}